\documentclass[12pt]{article}
\usepackage{amsthm}
\usepackage{amssymb,amsmath} 
\newtheorem{theorem}{Theorem}
\newtheorem{corollary}{Corollary}
\newtheorem{lemma}{Lemma}
\newtheorem{example}{Example}

\begin{document}

\title{Non-PORC behaviour of a class of descendant $p$-groups.}
\author{Marcus du\ Sautoy and Michael Vaughan-Lee\footnote{
The second author was partially supported by CIRM-FBK, Trento}\\
Mathematical Institute\\
24-29 St Giles\\
Oxford OX1 3LB\\
dusautoy@maths.ox.ac.uk\\
michael.vaughan-lee@chch.ox.ac.uk\\
}
\maketitle

\section{Introduction}

In \cite{duS1} the first author introduced the following nilpotent
group $G$ given by the presentation: 
\[
G=\left\langle 
\begin{array}{c}
x_{1},x_{2},x_{3},x_{4},x_{5},x_{6},y_{1},y_{2},y_{3}:\left[
x_{1},x_{4}\right] =y_{3},\left[ x_{1},x_{5}\right] =y_{1},\left[
x_{1},x_{6}\right] =y_{2} \\ 
\left[ x_{2},x_{4}\right] =y_{1},\left[ x_{2},x_{5}\right] =y_{3},\left[
x_{3},x_{4}\right] =y_{2},\left[ x_{3},x_{6}\right] =y_{1}
\end{array}
\right\rangle 
\]
where all other commutators are defined to be 1.

The group $G$ is a Hirsch length $9$, class two nilpotent group. This group turned out 
to have some fascinating properties especially in its local behaviour with respect to 
varying the prime $p$. In particular it was key to revealing that zeta functions that 
can be associated with nilpotent groups have a behaviour that mimics the arithmetic 
geometry of elliptic curves.

Given that this group has the arithmetic of the elliptic curve 
\[
E=Y^{2}-X^{3}+X
\] 
embedded into its structure it is interesting to explore other group theoretic features 
which reflect this arithmetic. The presentation can be refined to define a group $G_p$ which 
is a finite $p$-group of exponent $p$ and order $p^9$. It turns out that the automorphism 
group of $G_p$ depends very irregularly on $p$, again reflecting the arithmetic of the 
underlying elliptic curve. This impacts very interestingly on the number of immediate 
descendants of $G_p$. (These are the class 3 groups $K$ such that $K/\gamma _{3}(K)$ is
isomorphic to $G_p$.) Immediate descendants of $G_p$ are either of 
order $p^{10}$ or $p^{11}$. For $p>3$ the number of descendants of exponent $p$ with 
order $p^{10}$ is described by the following:

\begin{theorem} Let $D_p$ be the number of descendants of $G_p$ of order $p^{10}$ and
exponent $p$. Let $V_p$ be the number of solutions $(x,y)$ 
in $\mathbb{F}_{p}$ that satisfy $x^4+6x^2-3=0$ and $y^2=x^3-x$.
\begin{enumerate}
\item If $p=5 \mod{12}$ then $D_p=(p+1)^2/4+3$. 
\item If $p=7 \mod{12}$ then $D_p=(p+1)^2/2+2$.
\item If $p=11 \mod{12}$ then $D_p=(p+1)^2/6+(p+1)/3+2$.
\item If $p=1 \mod{12}$ and $V_p=0$ then $D_p=(p+1)^2/4+3$.
\item If $p=1 \mod{12}$ and $V_p\not=0$ then $D_p=(p-1)^2/36+(p-1)/3+4$.
\end{enumerate}
\end{theorem}

\begin{theorem}\label{Main Theorem}
There are infinitely many primes $p=1\mod{12}$ for which $V_p>0$. However there is 
no sub-congruence of $p=1\mod{12}$ for which $V_p>0$ for all $p$ in that sub-congruence class.
\end{theorem}

This theorem has an impact on Higman's PORC conjecture, which relates to the form of the
function $f(p,n)$ giving the number of non-isomorphic $p$-groups of order $p^n$.
(We will give a full statement of the conjecture and some of its history in Section 2.)

\begin{corollary}
The number of immediate descendants of $G_p$ of order $p^{10}$ and exponent $p$ is not PORC.
\end{corollary}

\begin{corollary}
The number of immediate descendants of $G_p$ of order $p^{10}$ is not PORC.
\end{corollary}

\begin{proof}
Let $E_p$ be the number of descendants of $G_p$ of order $p^{10}$ which do
not have exponent $p$. Then the total number of descendants of order $p^{10}$ is
$D_p+E_p$. When $p=1 \mod{12}$ and $V_p\not=0$ then $D_p$ has a lower value than
when $p=1 \mod{12}$ and $V_p=0$. Similarly, the value of $E_p$ is either the same
when $V_p\not=0$ as it is when $V_p=0$, or (more likely) it is also lower. So,
either way, the total number of descendants of $G_p$ of order $p^{10}$ is lower
when $p=1 \mod{12}$ and $V_p\not=0$ than it is when $p=1 \mod{12}$ and $V_p=0$. 
\end{proof}

The authors are very grateful to Jan Denef, Noam Elkies, Roger Heath-Brown and
Hans Opolka for a number of helpful conversations about number theory.

\section{Background}

In  \cite{GSS} Grunewald, Segal and Smith introduced the notion of the zeta function of a
group $G$: 
\[
\zeta _{G}^{\leq }(s)=\sum_{H\leq G}|G:H|^{-s}=\sum_{n=1}^{\infty
}a_{n}^{\leq }(G)n^{-s} 
\]
where $a_{n}^{\leq }(G)$ denotes the number of subgroups of index $n$ in $G.$
The definition of this zeta function as a sum over subgroups makes it look
like a non-commutative version of the Dedekind zeta function of a number
field. They proved that for finitely generated, torsion-free nilpotent
groups the global zeta function can be written as an Euler product of local
factors which are rational functions in $p^{-s}:$%
\begin{eqnarray*}
\zeta _{G}^{\leq }(s) &=&\prod_{p\text{ prime}}\zeta _{G,p}^{\leq }(s) \\
&=&\prod_{p\text{ prime}}Z_{p}^{\leq }(p,p^{-s})
\end{eqnarray*}
where for each prime $p,$ $\zeta _{G,p}^{\leq }(s)=\sum_{n=0}^{\infty
}a_{p^{n}}^{\leq }(G)p^{-ns}$ and $Z_{p}^{\leq }(X,Y)\in \mathbb{Q}(X,Y).$

Similar definitions and results were also obtained for the zeta 
function $\zeta _{G}^{\triangleleft }(s)$ counting normal subgroups.

One of the major questions raised in the paper \cite{GSS} is the variation
with $p$ of these local factors $Z_{p}^{\leq }(X,Y).$ Many of the examples
showed a uniform behaviour as the prime varied. For example, if $G$ is the
discrete Heisenberg group 
\[
G=\left( 
\begin{array}{lll}
1 & \mathbb{Z} & \mathbb{Z} \\ 
0 & 1 & \mathbb{Z} \\ 
0 & 0 & 1
\end{array}
\right) 
\]
then for all primes $p$%
\[
\zeta _{G,p}^{\leq }=\frac{(1-p^{3-3s})}{
(1-p^{-s})(1-p^{1-s})(1-p^{2-2s})(1-p^{3-2s})}. 
\]

However, if one takes the Heisenberg group with entries now from some
quadratic number field then it was shown in \cite{GSS} that the local
factors $Z_{p}^{\triangleleft }(X,Y)$ counting normal subgroups depend on
how the prime $p$ behaves in the quadratic number field. The authors 
of \cite{GSS} were led by such examples and the analogy with the Dedekind zeta
function of a number field to ask whether the local factors always
demonstrated a Cebotarev density type behaviour, depending on the behaviour
of primes in number fields. In particular they speculated in \cite{GSS} that
it was `plausible' that the following question has a positive answer:

\textbf{Question}\textsl{\ Let }$G$\textsl{\ be a finitely generated
nilpotent group and }$*\in \{\leq ,\triangleleft \}$. \textsl{Do there exist
finitely many rational functions }$W_{1}(X,Y),\ldots ,W_{r}(X,Y)\in \mathbb{Q}%
(X,Y)$ \textsl{such that for each prime} $p$ \textsl{there is an} $i$ 
\textsl{for which } 
\[
\zeta _{G,p}^{*}(s)=W_{i}(p,p^{-s})? 
\]

If the answer is `yes' we say that the local zeta functions $\zeta
_{G,p}^{*}(s)$ of $G$ are \emph{finitely uniform}. If there is one rational
function $W(X,Y)$ such that $\zeta _{G,p}^{*}(s)=W(p,p^{-s})$ for almost all
primes then we say that the local zeta functions $\zeta _{G,p}^{*}(s)$ of $G$
are\emph{\ uniform}.

Grunewald, Segal and Smith elevated this question to a conjecture in the
case that $G$ is a free nilpotent group. In \cite{GSS} they confirmed the
conjecture in the case that $G$ is a free nilpotent group of class 2.

The question of the behaviour of these local factors has gained extra
significance in the light of recent work of the first author on counting the 
number $f(p,n)$ of non-isomorphic finite $p$-groups that exist of order $p^{n}$.
Higman's PORC conjecture \cite{Higman-Enumerating2} asserts that for 
fixed $n$, the number $f(p,n)$ is given by a polynomial in $p$ whose coefficients
depend on the residue class of $p$ modulo some fixed integer $N,$ 
(\textbf{P}olynomial \textbf{O}n \textbf{R}esidue \textbf{C}lasses). 
In \cite{duS-ERA} and \cite{duS-pgroups} it is explained how this conjecture is directly
related to whether certain local zeta functions attached to free nilpotent
groups are finitely uniform. Higman \cite{Higman-Enumerating2} proves that
for each $p$ and $n$ the number of groups of order $p^n$ which have Frattini
subgroups which are central and elementary abelian is PORC. A. Evseev \cite{evsee} has 
extended Higman's result to groups where the Frattini subgroup is central 
(and not necessarily elementary abelian). For $n \leq 7$ Higman's conjecture is 
known to hold true (see \cite{newobvl} and \cite{obrienvl2}). 

The examples of Grunewald, Segal and Smith hinted that the behaviour of the
local factors as one varied the prime would be related to the behaviour of
primes in number fields. However the work of the first author with 
Grunewald \cite{duSG-compte abscissa} and \cite{duSG-Abscissa} shows that 
this first impression is misplaced. The behaviour is rather governed by a 
different question, namely how the number of points $\mod{p}$ on a variety 
varies with $p$.

In \cite{duSG-compte abscissa} and \cite{duSG-Abscissa}, the first author and 
Grunewald show that for each finitely generated nilpotent group $G$ there exists 
an explicit system of subvarieties $E_{i}$ ($i\in T,$ $T$ finite) of a 
variety $Y$ defined over  $\mathbb{Z}$ and, for each subset $I$ of $T,$ a rational 
function $W_{I}(X,Y)\in \mathbb{Q}(X,Y)$ such that for almost all primes $p$%
\[
\zeta _{G,p}^{*}(s)=\sum_{I\subset T}c_{I}(p)W_{I}(p,p^{-s}) 
\]
where 
\[
c_{I}(p)=\mathrm{card}\{a\in Y(\mathbb{F}_{p}):a\in E_{i}(\mathbb{F}_{p})\text{ if
and only if }i\in I\}. 
\]

So the analogy with the Dedekind zeta function of a number field is too
simplistic, rather it is the Weil zeta function of an algebraic variety over 
$\mathbb{Z}$ that offers a better analogy.

In contrast to the behaviour of primes in number fields, the number of 
points $\mod{p}$ on a variety can vary wildly with the prime $p$ and
certainly does not have a finitely uniform description.

\begin{example}
\label{number of points on elliptic curve}(\cite{ir}, 18.4) Let $E$ be the
elliptic curve $E=Y^{2}-X^{3}+X$. Put 
\[
\left| E(\mathbb{F}_{p})\right| =\left| \left\{ (x,y)\in \mathbb{F}%
_{p}^{2}:y^{2}-x^{3}+x=0\right\} \right| . 
\]
If $p=3\mod{4}$ then $\left| E(\mathbb{F}_{p})\right|=p.$ 
However if $p=1\mod{4}$ then 
\[
\left| E(\mathbb{F}_{p})\right| =p-2a, 
\]
where $p=a^{2}+b^{2}$ and $a+ib=1\mod{(2+2i)}$.
\end{example}

(Note that $\left| E(\mathbb{F}_{p})\right| $ is one less than the value $N_{p}$
given in \cite{ir}, 18.4 since $N_{p}$ counts the number of points on the
projective version of $E$. This includes one extra point at infinity not
counted in the affine coordinates.)

However, despite this theoretical advance which moves the problem into the
behaviour of varieties $\mod{p}$, it was not clear still whether exotic
varieties like elliptic curves could arise in the setting of zeta functions
of groups. It might be that the question of Grunewald, Segal and Smith would
still have a positive answer since the varieties that arise out of the
analysis of the first author and Grunewald were always rational where the number of
points $\mod{p}$ is uniform in $p$.

The group defined at the beginning of this paper turned out to be the first example 
of a nilpotent group $G$ whose zeta function depends on the behaviour $\mod{p}$ of the 
number of points on the elliptic curve $E=Y^{2}-X^{3}+X$. To see where the elliptic
curve is hiding in this presentation, take the determinant of the $3\times 3$
matrix $(a_{ij})$ with entries $a_{ij}=\left[ x_{i},x_{j+3}\right]$. In \cite{duS1} 
the group is shown to provide a negative answer to the question of Grunewald,\ Segal 
and Smith:

\begin{theorem}
The local zeta functions $\zeta _{G,p}^{\leq }(s)$ and $%
\zeta _{G,p}^{\triangleleft }(s)$ are not finitely uniform.
\end{theorem}

\section{Arithmetic Geometry}

In this section we prove Theorem \ref{Main Theorem}. For the whole of this section 
we assume $p$ is a prime with $p=1 \mod{12}$.

\begin{lemma} There exists $x,y$ in $\mathbb{F}_p$ such that $x^4+6x^2-3=0$ and $y^2=x^3-x$ if 
and only if there exists $y \in \mathbb{F}_p$ satisfying $y^8+360y^4-48=0$.
\end{lemma}

\begin{proof}
Let $x,y \in \mathbb{F}_p$ satisfy $x^4+6x^2-3=0$ and $y^2=x^3-x$. Substitute $x^3-x$ for
$y^2$ in $y^8+360y^4-48$ and use the identity $x^4+6x^2-3=0$ to see that $y^8+360y^4-48=0$.
Conversely, let $y$ be a root of $y^8+360y^4-48$ in $\mathbb{F}_p$ and 
let $x =-\frac{1}{208}(y^6+388y^2)$. Substituting this value for $x$ in $x^4+6x^2-3$ we see
that $x^4+6x^2-3=0$, and substituting this value for $x$ in $y^2-x^3+x$ we see 
that $y^2=x^3-x$.

Note that although the prime 13 divides 208, this does not affect the proof of Lemma 1, since
neither $x^4+6x^2-3$ nor $y^8+360y^4-48$ have roots in $\mathbb{F}_{13}$.
\end{proof}

So we are interested in for which $p$ does $y^8+360y^4-48=0$ have a solution in $\mathbb{F}_p$.
The splitting field of $y^8+360y^4-48$ over $\mathbb{Q}$ has degree 16, so adjoining one root 
of $y^8+360y^4-48$ to $\mathbb{Q}$ gives a field which is not even Galois 
let alone abelian. But if we adjoin a root of $y^8+360y^4-48$ to $\mathbb{Q}(i,\sqrt{3})$ then
we obtain the full splitting field. This splitting field has degree 4 over
$\mathbb{Q}(i,\sqrt{3})$, with Galois group isomorphic to $C_4$. 
This will be helpful in our analysis.

Since $p=1 \mod{12}$, 3 is a quadratic residue of $p$. Also $p$ can be
written as $p=a^{2}-12b^{2}$ with $a,b>0$. We can now establish the
following:

\begin{theorem}
$z^{4}+360z^{2}-48=0$ has a solution in $\mathbb{F}_{p}$ if and only if $a=1 \mod{3}$.
\end{theorem}

\begin{proof}
We use quadratic reciprocity in the number field $\mathbb{Q}(\sqrt{3})$. We
have $p=\pi \cdot \pi ^{\prime }$ in $\mathbb{Q}(\sqrt{3})$ with $\pi =a+2b%
\sqrt{3}$ (where $\pi ^{\prime }$ denotes the conjugate of $\pi $).%
\[
z^{4}+360z^{2}-48=(z^{2}-r)(z^{2}-s)
\]%
where $r=4\sqrt{3}(2-\sqrt{3})^{3}$ and $s=r^{\prime }$. So the question is
whether $r$ or $s$ can be a square mod $\pi $. Since $p=1 \mod{12}$, $-48$
is a square mod $p$, and hence $rs$ is a square mod $p$. So $r$ is a square
mod $\pi $ if and only if $s$ is a square mod $\pi $.

The condition for $r$ to be a square mod $\pi $ is given by the Law of
Quadratic Reciprocity for quadratic fields (see \cite{Lemmermeyer}). If $%
\alpha $ and $\beta $ are coprime elements of $\mathbb{Z}[\sqrt{3}]$ with odd 
norm, and if $\beta $ is irreducible, then the quadratic Legendre symbol 
$\left[ \frac{\alpha }{\beta }\right] $ 
is defined to be $+1$ or $-1$ depending on whether or not $\alpha $ 
is a square mod $\beta $. Eisenstein's quadratic reciprocity law
states that if $\alpha $, $\beta $, $\gamma $, $\delta $ are irreducible
elements with odd norm and if they satisfy $(\alpha ,\beta )=(\gamma ,\delta
)=(1)$ and $\alpha \equiv \gamma ,\beta \equiv \delta \mod{4\infty} $,
then 
\[
\left[ \frac{\alpha }{\beta }\right] \left[ \frac{\beta }{\alpha }\right] =%
\left[ \frac{\gamma }{\delta }\right] \left[ \frac{\delta }{\gamma }\right]. 
\]%
The notation $\alpha \equiv \gamma \mod{4\infty} $ means that $\alpha
=\gamma \mod{4}$ and that $\alpha $ and $\gamma $ have the same
signature, i.e. $($sign$\alpha ,$sign$\alpha ^{\prime })=($sign$\gamma ,$sign%
$\gamma ^{\prime })$. We want to know when $4\sqrt{3}(2-\sqrt{3})^{3}$, or
equivalently $\sqrt{3}(2-\sqrt{3})$, is a square mod $\pi $. So we take $%
\alpha =\pi $ and $\beta =\sqrt{3}(2-\sqrt{3})$. Note that $\alpha $ and $%
\beta $ are irreducible elements of $\mathbb{Z}[\sqrt{3}]$ with norms $p$
and $-3$. It follows that $\mathbb{Z}[\sqrt{3}]/(\beta )\cong \mathbb{F}_{3}$
and that $\left[ \frac{\alpha }{\beta }\right] =1$ if and only if $a=1 \mod{3}$. 
We establish Theorem 4 by showing that $\left[ \frac{\alpha }{\beta }%
\right] =\left[ \frac{\beta }{\alpha }\right] $.

If $b$ is even, then $\alpha =\xi ^{2} \mod{4}$, where $\xi =1$ or $\sqrt{%
3}$. So (by definition) $\alpha $ is primary with signature $(+1,+1)$ and $%
\left[ \frac{\alpha }{\beta }\right] =\left[ \frac{\beta }{\alpha }\right] $
by Corollary 12.9 of \cite{Lemmermeyer}.

So suppose that $b$ is odd. Then, depending on whether $a=1$ or $3 \mod{4}$,
we have $\alpha =5+2\sqrt{3} \mod{4}$ or $\alpha =11+2\sqrt{3} \mod{4}$%
. Accordingly, we take $\gamma =5+2\sqrt{3}$ or $\gamma =11+2\sqrt{3}$ and
take $\delta =\beta $. Note that $5+2\sqrt{3}$ and $11+2\sqrt{3}$ are
irreducible elements with norms 13 and 109, and that both have signature $%
(+1,+1)$. It is straightforward to check that in both cases $\left[ \frac{%
\gamma }{\delta }\right] =\left[ \frac{\delta }{\gamma }\right] =-1$, and so
Eisenstein's quadratic reciprocity law implies that $\left[ \frac{\alpha }{%
\beta }\right] =\left[ \frac{\beta }{\alpha }\right] $.
\end{proof}

We can now use the previous theorem to prove the following:

\begin{theorem}
There is no congruence class $p=c \mod{12d}$ with $c=1\mod{12}$ and $%
(c,d)=1$ for which $y^{8}+360y^{4}-48$ always has a root.
\end{theorem}

\begin{proof}
This follows provided we can show that there are primes $p=a^{2}-12b^{2}=c%
\mod{12d}$ with $a>0$ and $a=2\mod{3}$. By Dirichlet's Theorem, the
arithmetic progression $c+12nd$ $(n=1,2,\ldots )$ contains infinitely many
primes. Let $p$ be one of these primes, and write $p=a^{2}-12b^{2}$ with $a>0
$. If $a=2\mod{3}$ we are done. If not, 
consider the \textquotedblleft arithmetic progression\textquotedblright\ 
$-a+2b\sqrt{3}+12d(m+n\sqrt{3})$ with $m,n\in \mathbb{Z}$. From the 
$\mathbb{Q}(\sqrt{3})$ version of Dirichlet's theorem (see Rademacher 
\cite{rademacher35}), there is an irreducible element 
\[
\pi =-a+2b\sqrt{3}+12d(m+n\sqrt{3})
\]
for some $m,n\in \mathbb{Z}$, with $\pi >0$ and $\pi ^{\prime }>0$. Then%
\[
\pi \pi ^{\prime }=\left( -a+12dm+(2b+12dn)\sqrt{3}\right) \left(
-a+12dm-(2b+12dn)\sqrt{3}\right) 
\]%
is a rational prime 
\[
p=(-a+12dm)^{2}-12(b+6dn)^{2}=c\mod{12d},
\]%
with $-a+12dm>0$ and $(-a+12dm)=2\mod{3}$.
\end{proof}

The final piece of the jigsaw is the following:

\begin{theorem}
There are infinitely many solutions of $y^{8}+360y^{4}-48=0$ for $p=1\mod{12}$.
\end{theorem}

\begin{proof}
The splitting field of this polynomial has degree 16 over $\mathbb{Q}$, and
so the set of primes $p$ for which the polynomial splits over $\mathbb{F}_{p}
$ has Dirichlet density $\frac{1}{16}$ (see p.266 of Ireland and Rosen \cite%
{ir}). In particular, there are infinitely many such primes and they must
all be equal to $1\mod{12}$.
\end{proof}

\section{Counting the descendants of $G_{p}$}

We use the Lazard correspondence \cite{bourbaki} to count the immediate
descendants of $G_{p}$ of exponent $p$. This method was used in the
classification of groups of order $p^{6}$ \cite{newobvl} and $p^{7}$ \cite%
{obrienvl2}, and is explained in \cite{newobvl}. The Lazard correspondence
provides an isomorphism between the category of nilpotent Lie rings of order 
$p^{n}$ and nilpotency class at most $p-1$ and the category of $p$-groups of
order $p^{n}$ and class at most $p-1$. In particular, it gives an
isomorphism between the category of nilpotent Lie algebras of dimension $n$ 
over the field $\mathbb{F}_{p}$ and class at most $p-1$ and the category of
groups of exponent $p$ of order $p^{n}$ and class at most $p-1$. The Lie
algebra $L_{p}$ over $\mathbb{F}_{p}$ corresponding to the group $G_{p}$ has
a presentation on generators $x_{1},x_{2},\ldots ,x_{6},y_{1},y_{2},y\,_{3}$
with relations%
\[
\lbrack x_{1},x_{4}]=y_{3},\;[x_{1},x_{5}]=y_{1},\;[x_{1},x_{6}]=y_{2}, 
\]%
\[
\lbrack
x_{2},x_{4}]=y_{1},\;[x_{2},x_{5}]=y_{3},\;[x_{3},x_{4}]=y_{2},%
\;[x_{3},x_{6}]=y_{1}, 
\]%
and with all other Lie commutators trivial. Note that in this particular
case the presentation for the Lie algebra corresponding to $G_{p}$ is
identical to the presentation for $G_{p}$, though of course the commutators
have to be read as Lie commutators rather than as group commutators. This Lie
algebra is nilpotent of class 2 and of dimension 9, with $[L_{p},L_{p}]$
having dimension 3 and vector space basis $[x_{1},x_{4}]$, $[x_{1},x_{5}]$, $%
[x_{1},x_{6}]$. (Note that these basis elements for $[L_{p},L_{p}]$ are
equal to the defining generators $y_{3},y_{1},y_{2}$, but to avoid
notational conflict we will not use these three defining generators in the
following discussion.) For $p>3$ the immediate descendants of $L_{p}$
correspond under the Lazard correspondence to the immediate descendants of $%
G_{p}$ of exponent $p$.

It turns out that $L_{p}$ has immediate descendants of dimension 10 and 11,
and Theorem 1 is obtained by counting the immediate descendants of $L_{p}$
of dimension 10. A Lie algebra $A$ over $\mathbb{F}_{p}$ is (by definition)
an immediate descendant of $L_{p}$ if $A$ is nilpotent of class 3 and if $%
A/[A,A,A]\cong L_{p}$. We compute the immediate descendants as follows.
First we find the covering algebra for $L_{p}$. This is the largest Lie
algebra $M$ which is nilpotent of class 3 and contains an ideal $I$
satisfying the following properties:

\begin{enumerate}
\item $M/I\cong L_{p},$

\item $I\leq \lbrack M,M],$

\item $I$ is contained in the centre of $M.$
\end{enumerate}

The immediate descendants of $L_{p}$ are Lie algebras $M/J$, where $J$ is an
ideal of $M$ with $J<I$ and $J+[M,M,M]=I$. The trickiest part of classifying
the immediate descendants of $L_{p}$ is determining when two immediate
descendants $M/J$ and $M/K$ are isomorphic, and to solve this problem we
need to know the automorphism group of $L_{p}$.

\section{The automorphism group of $L_{p}$}

Let $V$ be the vector subspace of $L_{p}$ spanned by $%
x_{1},x_{2},x_{3},x_{4},x_{5},x_{6}$. It is sufficient to compute the
subgroup $G$ of the automorphism group of $L_{p}$ which maps $V$ onto $V$.
We claim that if $\left[ 
\begin{array}{cc}
\alpha & \beta \\ 
\gamma & \delta%
\end{array}%
\right] \in \,$GL$(2,p)$ then there is an automorphism in $G$ defined as
follows:%
\begin{eqnarray*}
x_{1} &\rightarrow &\alpha x_{1}+\beta x_{4}, \\
x_{2} &\rightarrow &\alpha x_{2}+\beta x_{5}, \\
x_{3} &\rightarrow &\alpha x_{3}+\beta x_{6}, \\
x_{4} &\rightarrow &\gamma x_{1}+\delta x_{4}, \\
x_{5} &\rightarrow &\gamma x_{2}+\delta x_{5}, \\
x_{6} &\rightarrow &\gamma x_{3}+\delta x_{6}.
\end{eqnarray*}%
Let $y_{i}$ be the image of $x_{i}$ under this map, for $i=1,2,3,4,5,6$. We
show that $y_{1},y_{2},\ldots ,y_{6}$ satisfy the defining relations of $%
L_{p}$. An important and useful property of $L_{p}$ is the following: if $%
1\leq i,j\leq 3$ then%
\[
\lbrack x_{3+i},x_{j}]=[x_{3+j},x_{i}]. 
\]%
We will regularly make use of this property without comment.

First consider $[y_{2},y_{1}]$.%
\begin{eqnarray*}
&&\lbrack y_{2},y_{1}] \\
&=&[\alpha x_{2}+\beta x_{5},\alpha x_{1}+\beta x_{4}] \\
&=&\alpha \beta \lbrack x_{2},x_{4}]+\alpha \beta \lbrack x_{5},x_{1}] \\
&=&0.
\end{eqnarray*}

The proofs that $[y_{3},y_{1}]=[y_{3},y_{2}]=0$ and that $[y_{i},y_{j}]=0$
for $i,j\in \{4,5,6\}$, are similar.

Now let $1\leq i,j\leq 3$. Then%
\begin{eqnarray*}
&&[y_{3+i},y_{j}] \\
&=&[\gamma x_{i}+\delta x_{3+i},\alpha x_{j}+\beta x_{3+j}] \\
&=&(\alpha \delta -\beta \gamma )[x_{3+i},x_{j}].
\end{eqnarray*}%
It follows immediately from this that%
\[
\lbrack y_{4},y_{1}]=[y_{5},y_{2}], 
\]%
\[
\lbrack y_{4},y_{3}]=[y_{6},y_{1}], 
\]%
\[
\lbrack y_{5},y_{1}]=[y_{4},y_{2}]=[y_{6},y_{3}], 
\]%
\[
\lbrack y_{5},y_{3}]=[y_{6},y_{2}]=0. 
\]%
So this map does define an automorphism of $L_{p}$.

The subspace of $V$ spanned by $x_{1},x_{2},x_{3}$ generates an abelian
subalgebra of $L_{p}$ of dimension 3, and it is fairly easy to check that
every three dimensional subspace of $V$ which generates an abelian
subalgebra of $L_{p}$ is the image of Sp$\langle x_{1},x_{2},x_{3}\rangle $
under one of the automorphisms described above. (See Section 6 below.) So,
modulo these automorphisms, it is sufficient to consider the subgroup $H\leq
G$ consisting of automorphisms which map Sp$\langle x_{1},x_{2},x_{3}\rangle 
$ to itself, and also map Sp$\langle x_{4},x_{5},x_{6}\rangle $ to itself.
So from now on we will look for automorphisms in $H$.

The action of GL$(2,p)$ described above gives automorphisms in $H$ of the
form%
\begin{eqnarray*}
x_{1} &\rightarrow &\alpha x_{1}, \\
x_{2} &\rightarrow &\alpha x_{2}, \\
x_{3} &\rightarrow &\alpha x_{3}, \\
x_{4} &\rightarrow &\delta x_{4}, \\
x_{5} &\rightarrow &\delta x_{5}, \\
x_{6} &\rightarrow &\delta x_{6}.
\end{eqnarray*}%
In addition there are automorphisms in $H$ defined by%
\begin{eqnarray*}
x_{1} &\rightarrow &-x_{1}, \\
x_{2} &\rightarrow &-x_{2}, \\
x_{3} &\rightarrow &x_{3}, \\
x_{4} &\rightarrow &-x_{4}, \\
x_{5} &\rightarrow &-x_{5}, \\
x_{6} &\rightarrow &x_{6},
\end{eqnarray*}%
and%
\begin{eqnarray*}
x_{1} &\rightarrow &ux_{1}, \\
x_{2} &\rightarrow &-ux_{2}, \\
x_{3} &\rightarrow &x_{3}, \\
x_{4} &\rightarrow &ux_{4}, \\
x_{5} &\rightarrow &-ux_{5}, \\
x_{6} &\rightarrow &x_{6},
\end{eqnarray*}%
where $u^{2}=-1$. (Of course the last of these can only occur when $p=1\mod{4}$.)

In addition, for some primes $p$ there are automorphisms of the form%
\[
\left[ 
\begin{array}{c}
x_{1} \\ 
x_{2} \\ 
x_{3}%
\end{array}%
\right] \longmapsto \left[ 
\begin{array}{ccc}
a & ab & ac \\ 
df & -f & -def \\ 
1 & d & e%
\end{array}%
\right] \left[ 
\begin{array}{c}
x_{1} \\ 
x_{2} \\ 
x_{3}%
\end{array}%
\right] 
\]%
and%
\[
\left[ 
\begin{array}{c}
x_{4} \\ 
x_{5} \\ 
x_{6}%
\end{array}%
\right] \longmapsto \left[ 
\begin{array}{ccc}
a & ab & ac \\ 
df & -f & -def \\ 
1 & d & e%
\end{array}%
\right] \left[ 
\begin{array}{c}
x_{4} \\ 
x_{5} \\ 
x_{6}%
\end{array}%
\right] . 
\]%
These automorphisms occur when we can solve the two equations%
\begin{eqnarray*}
d^{4}+6d^{2}-3 &=&0, \\
1-d^{2}+de^{2} &=&0
\end{eqnarray*}%
over $\mathbb{F}_{p}$. Then we let $a$ be a solution of $a^{2}=\pm \frac{%
(d^{2}-1)^{2}}{4d}$, and we set $b=\frac{3d+d^{3}}{1-d^{2}}$, $c=\allowbreak 
\frac{e\left( d^{2}+3\right) }{d^{2}-1}$, $f=\frac{d^{2}-1}{2da}$.

The equation $x^{2}+6x-3$ has roots $-3\pm \sqrt{12}$, and so there is no
solution to the equations unless $3$ is a quadratic residue modulo $p$.
Using quadratic reciprocity we see that 3 is a quadratic residue modulo $p$
if $p=\pm 1\mod{12}$.

The case $p=-1\mod{12}$ is straightforward. We need to find solutions to 
\[
d^{2}=-3\pm \sqrt{12}. 
\]
Since%
\[
(-3+\sqrt{12})(-3-\sqrt{12})=-3, 
\]%
which is \emph{not} a quadratic residue modulo $p$, we see that one of these
two equations has a solution and the other does not. So we have two
solutions $\pm d$ to the quartic equation. We now need to solve the equation%
\[
e^{2}=\frac{d^{2}-1}{\pm d}, 
\]%
and again, one of these equations has two solutions and the other has none.
So the two equations have exactly two solutions $d,\pm e$. For each of these
two solutions we obtain two possibilities for $a\,$, and then the given
values for $d,e,a$ determine $b,c,f$. So there are four automorphisms of
this form.

The case $p=1\mod{12}$ is much more complicated. In this case 
\[
(-3+\sqrt{12})(-3-\sqrt{12}) 
\]%
is a quadratic residue modulo $p$, and so either both the equations $%
d^{2}=-3\pm \sqrt{12}$ have solutions, or neither equation has a solution.
So there are either 0 or 4 solutions to $d^{4}+6d^{2}-3=0$. Suppose that we
have four solutions $\pm d_{1},\pm d_{2}$. Then we need to solve the
equations%
\[
e^{2}=\frac{d_{1}^{2}-1}{\pm d_{1}},\;e^{2}=\frac{d_{2}^{2}-1}{\pm d_{2}}. 
\]%
Since $-1$ is a quadratic residue modulo $p$ it is clear that $e^{2}=\frac{%
d_{1}^{2}-1}{\pm d_{1}}$ either has 4 solutions or none, and similarly $%
e^{2}=\frac{d_{2}^{2}-1}{\pm d_{2}}$ either has 4 solutions or none. Now%
\[
\frac{d_{1}^{2}-1}{d_{1}}\cdot \frac{d_{2}^{2}-1}{d_{2}}=\frac{(-4+\sqrt{12}%
)(-4-\sqrt{12})}{\sqrt{-3}}=\frac{4}{\sqrt{-3}}, 
\]%
and it turns out that $\sqrt{-3}$ is a square. This is because if $u^{2}=-1$
then 
\[
(\frac{1}{4}(1+u)(d^{3}+5d))^{4}=-3. 
\]%
So the equation $d^{4}+6d^{2}-3=0$ either has no solutions or four
solutions, and in the case when there are solutions then we either obtain no
solutions to the equations $1-d^{2}+de^{2}=0$, or we obtain a total of 8
solutions. The experimental evidence from looking at primes less than a
million indicates that the equation $d^{4}+6d^{2}-3=0$ has solutions for
approximately half the primes $p=1\mod{12}$, and that approximately half
of the primes $p=1\mod{12}$ which have solutions to $d^{4}+6d^{2}-3=0$
also have solutions to the equations $1-d^{2}+de^{2}=0$. Note that $d,e$ is
a solution to these two equations in $\mathbb{F}_{p}$ if and only if
$(x,y)=(d,de)$ is a solution to the two equations $x^4+6x^2-3=0$ 
and $y^2=x^3-x$. So, from Theorem 2 we
see that there are infinitely many primes $p=1\mod{12}$ for which the
two equations have solutions, but that there is no sub-congruence of $p=1%
\mod{12}$ such that there are solutions to the two equations for all $p$
in that sub-congruence class.

For each solution $d,e$ there are 4 solutions for $a$ with $a^{2}=\pm \frac{%
(d^{2}-1)^{2}}{4d}$. To see this note that to find 4 solutions for $a$ it is
sufficient that $-d$ be a square. Since $-d=\frac{1-d^{2}}{e^{2}}$ we need $%
1-d^{2}$ to be a square, and this is indeed the case since%
\[
4(1-d^{2})=4(1-d^{2})+(d^{4}+6d^{2}-3)=d^{4}+2d^{2}+1=(d^{2}+1)^{2}. 
\]%
So the four solutions for $a$ are $u\frac{(d^{2}+1)e}{4}$ where $u^{4}=1$.
The values of $b,c,f$ are determined by $d,e,a$. So there are 0 or 32
automorphisms of this form.

We give proofs that these are the only automorphisms in $H$ in Section 7.

\section{Abelian subalgebras of dimension 3}

As above we let $V$ be the vector subspace of $L_{p}$ spanned by $%
x_{1},x_{2},x_{3},x_{4},x_{5},x_{6}$. In this section we justify our claim
made above that any 3 dimensional subspace of $V$ which generates an abelian
subalgebra of $L_{p}$ has the form Sp$\langle \alpha x_{1}+\beta
x_{4},\alpha x_{2}+\beta x_{5},\alpha x_{3}+\beta x_{6}\rangle $ for some $%
\alpha ,\beta $. So let $W$ be such a subspace of $V$. Let $U=\,$Sp$\langle
x_{1},x_{2},x_{3}\rangle $.

First assume that $U\cap W\neq \{0\}$, and let $u\in U\cap W\backslash \{0\}$%
. Then $W$ must be a subspace of the centralizer of $u$ in $V$, $C_{V}(u)$.
We consider the possibilities for $C_{V}(u)$. First note that 
\begin{eqnarray*}
C_{V}(x_{2}) &=&\text{Sp}\langle x_{1},x_{2},x_{3},x_{6}\rangle , \\
C_{V}(x_{3}) &=&\text{Sp}\langle x_{1},x_{2},x_{3},x_{5}\rangle ,
\end{eqnarray*}%
and that if $\lambda \neq 0$ then%
\[
C_{V}(x_{2}+\lambda x_{3})=U. 
\]%
Next consider $C_{V}(x_{1}+dx_{2}+ex_{3})$. We have%
\begin{eqnarray*}
\lbrack x_{4},x_{1}+dx_{2}+ex_{3}]
&=&[x_{4},x_{1}]+d[x_{5},x_{1}]+e[x_{6},x_{1}], \\
\lbrack x_{5},x_{1}+dx_{2}+ex_{3}] &=&d[x_{4},x_{1}]+[x_{5},x_{1}], \\
\lbrack x_{6},x_{1}+dx_{2}+ex_{3}] &=&e[x_{5},x_{1}]+[x_{6},x_{1}].
\end{eqnarray*}%
It follows that $C_{V}(x_{1}+ax_{2}+bx_{3})=U$ unless%
\[
\det \left[ 
\begin{array}{ccc}
1 & d & e \\ 
d & 1 & 0 \\ 
0 & e & 1%
\end{array}%
\right] =1-d^{2}+de^{2}=0, 
\]%
in which case $C_{V}(x_{1}+dx_{2}+ex_{3})=\,$Sp$\langle
x_{1},x_{2},x_{3},dx_{4}-x_{5}-dex_{6}\rangle $. Since $W\leq C_{V}(u)$ we
see that either $W=U$, or $W\,$\ is a subspace of one of Sp$\langle
x_{1},x_{2},x_{3},x_{6}\rangle $, Sp$\langle x_{1},x_{2},x_{3},x_{5}\rangle $%
, Sp$\langle x_{1},x_{2},x_{3},dx_{4}-x_{5}-dex_{6}\rangle $. It follows
that $W$ has non-trivial intersection with Sp$\langle x_{1},x_{3}\rangle $.
Now $C_{V}(x_{1}+\lambda x_{3})=U$ (for any $\lambda $), and so if $W\neq U$
we must have $x_{3}\in W$. Similarly, using the fact that $W$ has
non-trivial intersection with Sp$\langle x_{1},x_{2}\rangle $, we see that
if $W\neq U$ then one of $x_{1}+x_{2}$, $x_{1}-x_{2}$, $x_{2}$ lies in $W$.
But this implies that one of $x_{1}+x_{2}+x_{3}$, $x_{1}-x_{2}+x_{3}$, $%
x_{2}+x_{3}$ lies in $W$. These three elements all have centralizers equal
to $U$, and so $W=U$.

Now assume the $U\cap W=\{0\}$. Then $W=$Sp$\langle
u_{1}+x_{4},u_{2}+x_{5},u_{3}+x_{6}\rangle $ for some $u_{1},u_{2},u_{3}\in
U $. Since $W$ is abelian we have%
\begin{eqnarray*}
\lbrack x_{4},u_{2}] &=&[x_{5},u_{1}], \\
\lbrack x_{4},u_{3}] &=&[x_{6},u_{1}], \\
\lbrack x_{5},u_{3}] &=&[x_{6},u_{2}],
\end{eqnarray*}%
and it is straightforward to show that this implies that for some $\lambda $
we have $u_{1}=\lambda x_{1}$, $u_{2}=\lambda x_{2}$, $u_{3}=\lambda x_{3}$.

This establishes our claim.

\section{Automorphisms in $H$}

We consider automorphisms of $L_{p}$ which map Sp$\langle
x_{1},x_{2},x_{3}\rangle $ to itself, and also map Sp$\langle
x_{4},x_{5},x_{6}\rangle $ to itself. These automorphisms take the form%
\[
\left[ \;%
\begin{array}{c}
x_{1} \\ 
x_{2} \\ 
x_{3}%
\end{array}%
\right] \rightarrow A\left[ \;%
\begin{array}{c}
x_{1} \\ 
x_{2} \\ 
x_{3}%
\end{array}%
\right] ,\;\left[ \;%
\begin{array}{c}
x_{4} \\ 
x_{5} \\ 
x_{6}%
\end{array}%
\right] \rightarrow B\left[ \;%
\begin{array}{c}
x_{4} \\ 
x_{5} \\ 
x_{6}%
\end{array}%
\right] 
\]%
where $A$ and $B$ are non-singular $3\times 3$ matrices over $\mathbb{F}_{p}$%
.

First we show that for automorphisms of this form we must have $A=\lambda B$
for some scalar $\lambda $.

So let $\theta $ be an automorphism of this form. Recall that 
\[
C_{V}(x_{2})=\text{Sp}\langle x_{1},x_{2},x_{3},x_{6}\rangle,
\]
and so $\theta x_{2}$ must also
be an element with centralizer of dimension 4. As we saw in Section 3, the
elements in Sp$\langle x_{1},x_{2},x_{3}\rangle $ with centralizers of
dimension 4 are scalar multiples of $x_{2}$ and $x_{3}$, and scalar
multiples of elements of the form $x_{1}+dx_{2}+ex_{3}$ where $%
1-d^{2}+de^{2}=0$. So $\theta x_{2}$ must be a scalar multiple of one of $%
x_{2}$, $x_{3}$ or $x_{1}+dx_{2}+ex_{3}$. This implies that $[\theta
x_{2},L_{p}]$ is one of the following:%
\begin{eqnarray*}
\lbrack x_{2},L_{p}] &=&\text{Sp}\langle \lbrack
x_{4},x_{1}],[x_{5},x_{1}]\rangle , \\
\lbrack x_{3},L_{p}] &=&\text{Sp}\langle \lbrack
x_{5},x_{1}],[x_{6},x_{1}]\rangle , \\
\lbrack x_{1}+dx_{2}+ex_{3},L_{p}] &=&\text{Sp}\langle \lbrack
x_{4},x_{1}]+d[x_{5},x_{1}]+e[x_{6},x_{1}],e[x_{5},x_{1}]+[x_{6},x_{1}]%
\rangle .
\end{eqnarray*}%
Note that these 2 dimensional subspaces are all different. In particular,
different solutions to the equation $1-d^{2}+de^{2}$ give different
subspaces. Similarly $\theta x_{5}$ must be a scalar multiple of one of $%
x_{5}$, $x_{6}$ or $x_{4}+dx_{5}+ex_{6}$, and so $[\theta x_{5},L_{p}]$ is
one of the following:%
\begin{eqnarray*}
\lbrack x_{5},L_{p}] &=&\text{Sp}\langle \lbrack
x_{4},x_{1}],[x_{5},x_{1}]\rangle , \\
\lbrack x_{6},L_{p}] &=&\text{Sp}\langle \lbrack
x_{5},x_{1}],[x_{6},x_{1}]\rangle , \\
\lbrack x_{4}+dx_{5}+[ex_{6},L_{p}] &=&\text{Sp}\langle \lbrack
x_{4},x_{1}]+d[x_{5},x_{1}]+e[x_{6},x_{1}],e[x_{5},x_{1}]+[x_{6},x_{1}]%
\rangle .
\end{eqnarray*}%
Now $[x_{2},L_{p}]=[x_{5},L_{p}]$, and so $[\theta x_{2},L_{p}]=[\theta
x_{5},L_{p}]$. This implies that one of three possibilities must arise:

\begin{enumerate}
\item $\theta x_{2}$ is a scalar multiple of $x_{2}$ and $\theta x_{5}$ is a
scalar multiple of $x_{5}$,

\item $\theta x_{2}$ is a scalar multiple of $x_{3}$ and $\theta x_{5}$ is a
scalar multiple of $x_{6}$,

\item $\theta x_{2}$ is a scalar multiple of $x_{1}+dx_{2}+ex_{3}$ and $%
\theta x_{5}$ is a scalar multiple of $x_{4}+dx_{4}+ex_{6}$ (with the same $%
d,e$).
\end{enumerate}

In other words, the second row of the matrix $A$ is a scalar multiple of the
second row of $B$. Similarly, the third row of $A$ is a scalar multiple of
the third row of $B$.

Now let%
\[
A=\left[ 
\begin{array}{ccc}
a_{11} & a_{12} & a_{13} \\ 
a_{21} & a_{22} & a_{23} \\ 
a_{31} & a_{32} & a_{33}%
\end{array}%
\right] ,\;B=\left[ 
\begin{array}{ccc}
b_{11} & b_{12} & b_{13} \\ 
b_{21} & b_{22} & b_{23} \\ 
b_{31} & b_{32} & b_{33}%
\end{array}%
\right] . 
\]%
The second and third rows of $B$ are scalar multiples of the second and
third rows of $A$, and so we can express $\left[ b_{11},b_{12},b_{13}\right] 
$ in the form%
\[
\lambda \left[ a_{11},a_{12},a_{13}\right] +\mu \left[ b_{21},b_{22},b_{23}%
\right] +\nu \left[ b_{31},b_{32},b_{33}\right] 
\]%
for some $\lambda ,\mu ,\nu $. It is a property of the algebra $L_{p}$ that
for any scalars $a,b,c,d,e,f$,%
\[
\lbrack
ax_{4}+bx_{5}+cx_{6},dx_{1}+ex_{2}+fx_{3}]=[dx_{4}+ex_{5}+fx_{6},ax_{1}+bx_{2}+cx_{3}]. 
\]%
It follows that%
\begin{eqnarray*}
&&[\theta x_{4},\theta x_{2}] \\
&=&\lambda \lbrack a_{11}x_{4}+a_{12}x_{5}+a_{13}x_{6},\theta x_{2}]+\mu
\lbrack \theta x_{5},\theta x_{2}]+\nu \lbrack \theta x_{6},\theta x_{2}] \\
&=&\lambda \lbrack
a_{11}x_{4}+a_{12}x_{5}+a_{13}x_{6},a_{21}x_{1}+a_{22}x_{2}+a_{23}x_{3}]+\mu
\lbrack \theta x_{5},\theta x_{2}]\text{ since }[x_{6},x_{2}]=0 \\
&=&\lambda \lbrack
a_{21}x_{4}+a_{22}x_{5}+a_{23}x_{6},a_{11}x_{1}+a_{12}x_{2}+a_{13}x_{3}]+\mu
\lbrack \theta x_{5},\theta x_{2}].
\end{eqnarray*}%
Now $a_{21}x_{4}+a_{22}x_{5}+a_{23}x_{6}$ is a scalar multiple of $\theta
x_{5}$, and so 
\[
\lambda \lbrack
a_{21}x_{4}+a_{22}x_{5}+a_{23}x_{6},a_{11}x_{1}+a_{12}x_{2}+a_{13}x_{3}] 
\]%
is a non-trivial scalar multiple of $[\theta x_{5},\theta x_{1}]=[\theta
x_{4},\theta x_{2}]$. On the other hand, $[\theta x_{5},\theta x_{2}]$ and $%
[\theta x_{4},\theta x_{2}]$ are linearly independent, and so we must have $%
\mu =0$. Similarly considering $[\theta x_{4},\theta x_{3}]$ we see that $%
\nu =0$. So the rows of $B$ are all scalar multiples of the rows of $A$.

We may now assume that%
\begin{eqnarray*}
\lbrack b_{11},b_{12},b_{13}] &=&\lambda \lbrack a_{11},a_{12},a_{13}], \\
\lbrack b_{21},b_{22},b_{23}] &=&\mu \lbrack a_{21},a_{22},a_{23}], \\
\lbrack b_{31},b_{32},b_{33}] &=&\nu \lbrack a_{31},a_{32},a_{33}]
\end{eqnarray*}%
for some $\lambda ,\mu ,\nu $. But then the relation $%
[x_{5},x_{1}]=[x_{4},x_{2}]$ implies that $\lambda =\mu $, and the relation $%
[x_{6},x_{1}]=[x_{4},x_{3}]$ implies that $\lambda =\nu $. So $B=\lambda A$,
as claimed.

Composing $\theta $ with an automorphism of the form%
\begin{eqnarray*}
x_{1} &\rightarrow &\alpha x_{1}, \\
x_{2} &\rightarrow &\alpha x_{2}, \\
x_{3} &\rightarrow &\alpha x_{3}, \\
x_{4} &\rightarrow &\delta x_{4}, \\
x_{5} &\rightarrow &\delta x_{5}, \\
x_{6} &\rightarrow &\delta x_{6},
\end{eqnarray*}%
we may assume that $A=B$, and that $\theta x_{3}$ equals $x_{2}$ or $x_{3}$
or $x_{1}+dx_{2}+ex_{3}$ for some solution of $1-d^{2}+de^{2}=0$.

First, we show that the possibility $\theta x_{3}=x_{2}$ never arises.
Suppose, to the contrary, that $\theta x_{3}=x_{2}$. The relation $%
[x_{5},x_{3}]=0$ implies that $\theta x_{5}=\lambda x_{6}$ for some $\lambda 
$. The condition $A=B$ implies that $\theta x_{2}=\lambda x_{3}$, $\theta
x_{6}=x_{5}$. Let $\theta x_{1}=ax_{1}+bx_{2}+cx_{3}$. Then%
\[
\lbrack \theta x_{5},\theta x_{1}]=\lambda c[x_{5},x_{1}]+\lambda
a[x_{6},x_{1}] 
\]%
and%
\[
\lbrack \theta x_{6},\theta x_{3}]=[x_{5},x_{2}]=[x_{4},x_{1}]. 
\]%
However this conflicts with the relation $[x_{5},x_{1}]=[x_{6},x_{3}]$, and
so $\theta x_{3}=x_{2}$ cannot arise.

Next consider the possibility that $\theta x_{3}=x_{3}$. Then we must have $%
\theta x_{2}=\lambda x_{2}$ for some $\lambda $. This gives $\theta
x_{5}=\lambda x_{5}$, $\theta x_{6}=x_{6}$. Let $\theta
x_{1}=ax_{1}+bx_{2}+cx_{3}$. Then%
\[
\lbrack \theta x_{5},\theta x_{1}]=\lambda b[x_{4},x_{1}]+\lambda
a[x_{5},x_{1}] 
\]%
and%
\[
\lbrack \theta x_{6},\theta x_{3}]=[x_{6},x_{3}]=[x_{5},x_{1}]. 
\]%
So the relation $[x_{5},x_{1}]=[x_{6},x_{3}]$ implies that $a=\lambda ^{-1}$%
, $b=0$. This gives%
\[
\lbrack \theta x_{4},\theta x_{1}]=\lambda
^{-2}[x_{4},x_{1}]+c^{2}[x_{5},x_{1}]+2\lambda ^{-1}c[x_{6},x_{1}] 
\]%
and%
\[
\lbrack \theta x_{5},\theta x_{2}]=\lambda ^{2}[x_{5},x_{2}]=\lambda
^{2}[x_{4},x_{1}]. 
\]%
So the relation $[x_{4},x_{1}]=[x_{5},x_{2}]$ gives $\lambda ^{4}=1$ and $%
c=0 $. So we have%
\[
A=B=\left[ 
\begin{array}{ccc}
\lambda ^{-1} & 0 & 0 \\ 
0 & \lambda & 0 \\ 
0 & 0 & 1%
\end{array}%
\right] 
\]%
where $\lambda ^{4}=1$.

Finally consider the possibility that $\theta x_{3}=x_{1}+dx_{2}+ex_{3}$ for
some $d,e$ satisfying $1-d^{2}+de^{2}=0$. The relation $[x_{5},x_{3}]=0$
implies that $\theta x_{5}=dfx_{4}-fx_{5}-defx_{6}$ for some non-zero $f$.
The assumption that $A=B$ implies that $\theta x_{2}=dfx_{1}-fx_{2}-defx_{3}$%
, $\theta x_{6}=x_{4}+dx_{5}+ex_{6}$. Let $\theta x_{1}=ax_{1}+bx_{2}+cx_{3}$%
.

We first show that $a\neq 0$. Suppose to the contrary that $a=0$, so that $%
\theta x_{1}=bx_{2}+cx_{3}$ and $\theta x_{4}=bx_{5}+cx_{6}$. Then computing 
$[\theta x_{4},\theta x_{1}]$ and $[\theta x_{5},\theta x_{2}]$ we see that
the relation $[x_{4},x_{1}]=[x_{5},x_{2}]$ gives $d^{2}ef=0$. Since $f\neq 0$
and $d$ cannot equal 0, this implies that $e=0$, and hence that $d=\pm 1$.
But now computing $[\theta x_{5},\theta x_{1}]$ and $[\theta x_{6},\theta
x_{3}]$ we see that the relation $[x_{5},x_{1}]=[x_{6},x_{3}]$ gives $%
-fb=1+d^{2}=2$, $dfb=2d$, $dfc=0$. However the first two of these three
relations are incompatible, and so $a=0$ is impossible.

This means that we can take $\theta x_{1}=ax_{1}+abx_{2}+acx_{3}$, $\theta
x_{4}=ax_{4}+abx_{5}+acx_{6}$ for some $a,b,c$ with $a\neq 0$. Thus%
\[
A=B=\left[ 
\begin{array}{ccc}
a & ab & ac \\ 
df & -f & -def \\ 
1 & d & e%
\end{array}%
\right] . 
\]%
The relations $[x_{4},x_{1}]=[x_{5},x_{2}]$ and $[x_{5},x_{1}]=[x_{6},x_{3}]$
now give six equations which $a,b,c,d,e$ must satisfy:

\begin{eqnarray}
a^{2}(1+b^{2}) &=&f^{2}(1+d^{2}), \\
a^{2}(2b+c^{2}) &=&f^{2}(d^{2}e^{2}-2d),  \nonumber \\
a^{2}c &=&-d^{2}ef^{2},  \nonumber \\
af(d-b) &=&1+d^{2},  \nonumber \\
af(bd-cde-1) &=&2d+e^{2},  \nonumber \\
adf(c-e) &=&2e.  \nonumber
\end{eqnarray}%
Since $af\neq 0$, the last three equations above give%
\begin{eqnarray*}
(1+d^{2})(bd-cde-1)-(d-b)(2d+e^{2}) &=&0, \\
(1+d^{2})d(c-e)-(d-b)2e &=&0.
\end{eqnarray*}%
Multiplying the second of these two equations by $e$, and then adding to the
first, we obtain%
\[
(1+d^{2})(bd-1-de^{2})-(d-b)(2d+3e^{2})=0. 
\]%
Multiplying this equation by $d$, and then using the relation $%
1-d^{2}+de^{2}=0$ to eliminate $de^{2}$ we obtain%
\[
\left( b-d\right) \left( d^{4}+6d^{2}-3\right) =0. 
\]%
Now $b=d$ is impossible, because if $b=d$ then the equation $af(d-b)=1+d^{2}$
gives $d^{2}=-1$, which would imply that $A$ is singular. So we must have%
\[
d^{4}+6d^{2}-3=0. 
\]

The equation $(1+d^{2})d(c-e)-(d-b)2e=0\allowbreak $ gives $c=\frac{%
d^{3}e-2be+3de}{d(1+d^{2})}$. Since $a$ and $f$ are both non-zero, the first
and third equations from (1) give%
\[
(1+d^{2})c+(1+b^{2})d^{2}e=0. 
\]%
Substituting $\frac{d^{3}e-2be+3de}{d(1+d^{2})}$ for $c$ in this equation we
obtain

\[
e\left( b^{2}d^{3}-2b+2d^{3}+3d\right) =0. 
\]%
Now $e\neq 0$, for if $e=0$ then the equation $1-d^{2}+de^{2}=0$ implies
that $d=\pm 1$, which is incompatible with the equation $d^{4}+6d^{2}-3=0$.
So%
\begin{equation}
b^{2}d^{3}-2b+2d^{3}+3d=0.
\end{equation}

The second and third equations from (1) give%
\[
(d^{2}e^{2}-2d)c+(2b+c^{2})d^{2}e=0 
\]%
Substituting $\frac{d^{3}e-2be+3de}{d(1+d^{2})}$ for $c$, and then
substituting $\frac{d^{2}-1}{d}$ for $e^{2}$ we obtain%
\[
\left( -b+3d+bd^{2}+d^{3}\right) \left( 2b-3d+d^{5}\right) =0. 
\]%
This gives $b=\frac{3d-d^{5}}{2}$ or $b=\frac{d^{3}+3d}{1-d^{2}}$. However,
if we substitute $\frac{3d-d^{5}}{2}$ for $b$ in (2) we obtain

\[
d^{3}\left( d^{2}+1\right) ^{2}\left( -d^{6}+2d^{4}+3d^{2}-8\right) =0 
\]%
Now we know that $d\neq 0$, $d^{2}+1\neq 0$, $d^{4}+6d^{2}-3=0$. The
greatest common divisor of $d^{4}+6d^{2}-3$ and $-d^{6}+2d^{4}+3d^{2}-8$ is
1, and so this is impossible. So $b=\frac{d^{3}+3d}{1-d^{2}}$.

Substituting this value for $b$ into our expression for $c$ we obtain $c=%
\frac{e(d^{2}+3)}{d^{2}-1}$. Also, substituting this value of $b$ into the
fourth equation from (1), we obtain $f=\frac{d^{2}-1}{2da}$. Substituting
these values for $b$ and $f$ into the first equation from (1) we obtain%
\[
a^{4}=\frac{\left( d^{2}-1\right) ^{4}}{4d^{2}\left( d^{4}+6d^{2}+1\right) }=%
\frac{\left( d^{2}-1\right) ^{4}}{16d^{2}}. 
\]%
So, as we showed in Section 5, the solutions for $a$ are $a=u\frac{(d^{2}+1)e%
}{4}$ for any $u$ with $u^{4}=1$.

It is straightforward to verify that with these values of $a,b,c,d,e,f$ $\ $%
then $\theta x_{1}$, $\theta x_{2}$, $\ldots ,\theta x_{6},$ satisfy the
defining relations of $L_{p}$ provided $1-d^{2}+de^{2}=0$ and $%
d^{4}+6d^{2}-3=0$. To see this note that the property that 
\[
\lbrack \alpha x_{4}+\beta x_{5}+\gamma x_{6},\delta x_{1}+\varepsilon
x_{2}+\zeta x_{3}]=[\delta x_{4}+\varepsilon x_{5}+\zeta x_{6},\alpha
x_{1}+\beta x_{2}+\gamma x_{3}] 
\]%
for all $\alpha ,\beta ,\gamma ,\delta ,\varepsilon ,\zeta $ implies that%
\[
\lbrack \theta x_{4},\theta x_{1}]=[\theta x_{5},\theta x_{2}], 
\]%
\[
\lbrack \theta x_{4},\theta x_{3}]=[\theta x_{6},\theta x_{1}], 
\]%
\[
\lbrack \theta x_{5},\theta x_{1}]=[\theta x_{4},\theta x_{2}]. 
\]%
Also, $\theta x_{2}$ and $\theta x_{5}$ were chosen so that%
\[
\lbrack \theta x_{5},\theta x_{3}]=[\theta x_{6},\theta x_{2}]=0, 
\]%
and the relations%
\[
\lbrack \theta x_{i},\theta x_{j}]=0\text{ for }i,j\in \{1,2,3\}, 
\]%
\[
\lbrack \theta x_{i},\theta x_{j}]=0\text{ for }i,j\in \{4,5,6\} 
\]%
follow from the fact that $\theta x_{1},\theta x_{2},\theta x_{3}\in \,$Sp$%
\langle x_{1},x_{2},x_{3}\rangle $ and $\theta x_{4},\theta x_{5},\theta
x_{6}\in \,$Sp$\langle x_{4},x_{5},x_{6}\rangle $. So we only need to check
the relations $[\theta x_{4},\theta x_{1}]=[\theta x_{5},\theta x_{2}]$ and $%
[\theta x_{5},\theta x_{1}]=[\theta x_{6},\theta x_{3}]$, and the six
equations (1) ensure that these are satisfied. So we only need to check that 
$a,b,c,d,e,f$ satisfy the equations (1), and this is straightforward.

\section{The covering algebra}

To obtain the covering algebra for $L_{p}$ we need the following defining
relations for $L_{p}$ as a 9 dimensional Lie algebra with vector space basis 
$x_{1},x_{2},\ldots ,x_{9}$.%
\[
\lbrack x_{2},x_{1}]=0, 
\]%
\[
\lbrack x_{3},x_{1}]=0, 
\]%
\[
\lbrack x_{3},x_{2}]=0, 
\]%
\[
\lbrack x_{4},x_{1}]=x_{7}, 
\]%
\[
\lbrack x_{4},x_{2}]=x_{8}, 
\]%
\[
\lbrack x_{4},x_{3}]=x_{9}, 
\]%
\[
\lbrack x_{5},x_{1}]=x_{8}, 
\]%
\[
\lbrack x_{5},x_{2}]=x_{7}, 
\]%
\[
\lbrack x_{5},x_{3}]=0, 
\]%
\[
\lbrack x_{5},x_{4}]=0, 
\]%
\[
\lbrack x_{6},x_{1}]=x_{9}, 
\]%
\[
\lbrack x_{6},x_{2}]=0, 
\]%
\[
\lbrack x_{6},x_{3}]=x_{8}, 
\]%
\[
\lbrack x_{6},x_{4}]=0, 
\]%
\[
\lbrack x_{6},x_{5}]=0, 
\]%
\[
\lbrack x_{7},x_{1}]=0, 
\]%
\[
\lbrack x_{7},x_{2}]=0, 
\]%
\[
\lbrack x_{7},x_{3}]=0, 
\]%
\[
\lbrack x_{7},x_{4}]=0, 
\]%
\[
\lbrack x_{7},x_{5}]=0, 
\]%
\[
\lbrack x_{7},x_{6}]=0, 
\]%
\[
\lbrack x_{8},x_{1}]=0, 
\]%
\[
\lbrack x_{8},x_{2}]=0, 
\]%
\[
\lbrack x_{8},x_{3}]=0, 
\]%
\[
\lbrack x_{8},x_{4}]=0, 
\]%
\[
\lbrack x_{8},x_{5}]=0, 
\]%
\[
\lbrack x_{8},x_{6}]=0, 
\]%
\[
\lbrack x_{9},x_{1}]=0, 
\]%
\[
\lbrack x_{9},x_{2}]=0, 
\]%
\[
\lbrack x_{9},x_{3}]=0, 
\]%
\[
\lbrack x_{9},x_{4}]=0, 
\]%
\[
\lbrack x_{9},x_{5}]=0, 
\]%
\[
\lbrack x_{9},x_{6}]=0. 
\]

This presentation has 33 relations, but the relations $[x_{4},x_{1}]=x_{7}$, 
$[x_{4},x_{2}]=x_{8}$, $[x_{4},x_{3}]=x_{9}$ are taken to be the definitions
of $x_{7},x_{8},x_{9}$. We introduce 30 additional generators $%
x_{10},x_{11},\ldots ,x_{39}$ corresponding to the 30 relations which are 
\emph{not} definitions, and add them as \textquotedblleft
tails\textquotedblright\ to these relations. This gives the following
presentation for the covering algebra:%
\[
\lbrack x_{2},x_{1}]=x_{28}, 
\]%
\[
\lbrack x_{3},x_{1}]=x_{29}, 
\]%
\[
\lbrack x_{3},x_{2}]=x_{30}, 
\]%
\[
\lbrack x_{4},x_{1}]=x_{7}, 
\]%
\[
\lbrack x_{4},x_{2}]=x_{8}, 
\]%
\[
\lbrack x_{4},x_{3}]=x_{9}, 
\]%
\[
\lbrack x_{5},x_{1}]=x_{8}+x_{31}, 
\]%
\[
\lbrack x_{5},x_{2}]=x_{7}+x_{32}, 
\]%
\[
\lbrack x_{5},x_{3}]=x_{33}, 
\]%
\[
\lbrack x_{5},x_{4}]=x_{34}, 
\]%
\[
\lbrack x_{6},x_{1}]=x_{9}+x_{35}, 
\]%
\[
\lbrack x_{6},x_{2}]=x_{36}, 
\]%
\[
\lbrack x_{6},x_{3}]=x_{8}+x_{37}, 
\]%
\[
\lbrack x_{6},x_{4}]=x_{38}, 
\]%
\[
\lbrack x_{6},x_{5}]=x_{39}, 
\]%
\[
\lbrack x_{7},x_{1}]=x_{10}, 
\]%
\[
\lbrack x_{7},x_{2}]=x_{11}, 
\]%
\[
\lbrack x_{7},x_{3}]=x_{12}, 
\]%
\[
\lbrack x_{7},x_{4}]=x_{13}, 
\]%
\[
\lbrack x_{7},x_{5}]=x_{14}, 
\]%
\[
\lbrack x_{7},x_{6}]=x_{15}, 
\]%
\[
\lbrack x_{8},x_{1}]=x_{16}, 
\]%
\[
\lbrack x_{8},x_{2}]=x_{17}, 
\]%
\[
\lbrack x_{8},x_{3}]=x_{18}, 
\]%
\[
\lbrack x_{8},x_{4}]=x_{19}, 
\]%
\[
\lbrack x_{8},x_{5}]=x_{20}, 
\]%
\[
\lbrack x_{8},x_{6}]=x_{21}, 
\]%
\[
\lbrack x_{9},x_{1}]=x_{22}, 
\]%
\[
\lbrack x_{9},x_{2}]=x_{23}, 
\]%
\[
\lbrack x_{9},x_{3}]=x_{24}, 
\]%
\[
\lbrack x_{9},x_{4}]=x_{25}, 
\]%
\[
\lbrack x_{9},x_{5}]=x_{26}, 
\]%
\[
\lbrack x_{9},x_{6}]=x_{27}, 
\]%
where in addition we also have relations implying that the tails are all
central. We now need to enforce the Jacobi identity%
\[
\lbrack x_{i},x_{j},x_{k}]+[x_{j},x_{k},x_{i}]+[x_{k},x_{i},x_{j}]=0 
\]%
for all $i,j,k$ with $1\leq k<j<i\leq 6$. This gives 20 Jacobi relations,
and we evaluate $[x_{i},x_{j},x_{k}]+[x_{j},x_{k},x_{i}]+[x_{k},x_{i},x_{j}]$
in each case.%
\[
\lbrack x_{3},x_{2},x_{1}]+[x_{2},x_{1},x_{3}]+[x_{1},x_{3},x_{2}]=0 
\]%
\[
\lbrack
x_{4},x_{2},x_{1}]+[x_{2},x_{1},x_{4}]+[x_{1},x_{4},x_{2}]=x_{16}-x_{11} 
\]%
\[
\lbrack
x_{4},x_{3},x_{1}]+[x_{3},x_{1},x_{4}]+[x_{1},x_{4},x_{3}]=x_{22}-x_{12} 
\]%
\[
\lbrack
x_{4},x_{3},x_{2}]+[x_{3},x_{2},x_{4}]+[x_{2},x_{4},x_{3}]=x_{23}-x_{18} 
\]%
\[
\lbrack
x_{5},x_{2},x_{1}]+[x_{2},x_{1},x_{5}]+[x_{1},x_{5},x_{2}]=x_{10}-x_{17} 
\]%
\[
\lbrack x_{5},x_{3},x_{1}]+[x_{3},x_{1},x_{5}]+[x_{1},x_{5},x_{3}]=-x_{18} 
\]%
\[
\lbrack x_{5},x_{3},x_{2}]+[x_{3},x_{2},x_{5}]+[x_{2},x_{5},x_{3}]=-x_{12} 
\]%
\[
\lbrack
x_{5},x_{4},x_{1}]+[x_{4},x_{1},x_{5}]+[x_{1},x_{5},x_{4}]=x_{14}-x_{19} 
\]%
\[
\lbrack
x_{5},x_{4},x_{2}]+[x_{4},x_{2},x_{5}]+[x_{2},x_{5},x_{4}]=x_{20}-x_{13} 
\]%
\[
\lbrack x_{5},x_{4},x_{3}]+[x_{4},x_{3},x_{5}]+[x_{3},x_{5},x_{4}]=x_{26} 
\]%
\[
\lbrack x_{6},x_{2},x_{1}]+[x_{2},x_{1},x_{6}]+[x_{1},x_{6},x_{2}]=-x_{23} 
\]%
\[
\lbrack
x_{6},x_{3},x_{1}]+[x_{3},x_{1},x_{6}]+[x_{1},x_{6},x_{3}]=x_{16}-x_{24} 
\]%
\[
\lbrack x_{6},x_{3},x_{2}]+[x_{3},x_{2},x_{6}]+[x_{2},x_{6},x_{3}]=x_{17} 
\]%
\[
\lbrack
x_{6},x_{4},x_{1}]+[x_{4},x_{1},x_{6}]+[x_{1},x_{6},x_{4}]=x_{15}-x_{25} 
\]%
\[
\lbrack x_{6},x_{4},x_{2}]+[x_{4},x_{2},x_{6}]+[x_{2},x_{6},x_{4}]=x_{21} 
\]%
\[
\lbrack
x_{6},x_{4},x_{3}]+[x_{4},x_{3},x_{6}]+[x_{3},x_{6},x_{4}]=x_{27}-x_{19} 
\]%
\[
\lbrack
x_{6},x_{5},x_{1}]+[x_{5},x_{1},x_{6}]+[x_{1},x_{6},x_{5}]=x_{21}-x_{26} 
\]%
\[
\lbrack x_{6},x_{5},x_{2}]+[x_{5},x_{2},x_{6}]+[x_{2},x_{6},x_{5}]=x_{15} 
\]%
\[
\lbrack x_{6},x_{5},x_{3}]+[x_{5},x_{3},x_{6}]+[x_{3},x_{6},x_{5}]=-x_{20} 
\]%
\[
\lbrack x_{6},x_{5},x_{4}]+[x_{5},x_{4},x_{6}]+[x_{4},x_{6},x_{5}]=0 
\]

So the Jacobi relations give the following:%
\begin{eqnarray*}
x_{10}
&=&x_{12}=x_{13}=x_{15}=x_{17}=x_{18}=x_{20}=x_{21}=x_{22}=x_{23}=x_{25}=x_{26}=0,
\\
x_{11} &=&x_{16}=x_{24}, \\
x_{14} &=&x_{19}=x_{27}.
\end{eqnarray*}%
If we enforce these relations, and relabel the generators, then we obtain
the following presentation for the covering algebra.%
\[
\lbrack x_{2},x_{1}]=x_{12}, 
\]%
\[
\lbrack x_{3},x_{1}]=x_{13}, 
\]%
\[
\lbrack x_{3},x_{2}]=x_{14}, 
\]%
\[
\lbrack x_{4},x_{1}]=x_{7}, 
\]%
\[
\lbrack x_{4},x_{2}]=x_{8}, 
\]%
\[
\lbrack x_{4},x_{3}]=x_{9}, 
\]%
\[
\lbrack x_{5},x_{1}]=x_{8}+x_{15}, 
\]%
\[
\lbrack x_{5},x_{2}]=x_{7}+x_{16}, 
\]%
\[
\lbrack x_{5},x_{3}]=x_{17}, 
\]%
\[
\lbrack x_{5},x_{4}]=x_{18}, 
\]%
\[
\lbrack x_{6},x_{1}]=x_{9}+x_{19}, 
\]%
\[
\lbrack x_{6},x_{2}]=x_{20}, 
\]%
\[
\lbrack x_{6},x_{3}]=x_{8}+x_{21}, 
\]%
\[
\lbrack x_{6},x_{4}]=x_{22}, 
\]%
\[
\lbrack x_{6},x_{5}]=x_{23}, 
\]%
\[
\lbrack x_{7},x_{1}]=0, 
\]%
\[
\lbrack x_{7},x_{2}]=x_{10}, 
\]%
\[
\lbrack x_{7},x_{3}]=0, 
\]%
\[
\lbrack x_{7},x_{4}]=0, 
\]%
\[
\lbrack x_{7},x_{5}]=x_{11}, 
\]%
\[
\lbrack x_{7},x_{6}]=0, 
\]%
\[
\lbrack x_{8},x_{1}]=x_{10}, 
\]%
\[
\lbrack x_{8},x_{2}]=0, 
\]%
\[
\lbrack x_{8},x_{3}]=0, 
\]%
\[
\lbrack x_{8},x_{4}]=x_{11}, 
\]%
\[
\lbrack x_{8},x_{5}]=0, 
\]%
\[
\lbrack x_{8},x_{6}]=0, 
\]%
\[
\lbrack x_{9},x_{1}]=0, 
\]%
\[
\lbrack x_{9},x_{2}]=0, 
\]%
\[
\lbrack x_{9},x_{3}]=x_{10}, 
\]%
\[
\lbrack x_{9},x_{4}]=0, 
\]%
\[
\lbrack x_{9},x_{5}]=0, 
\]%
\[
\lbrack x_{9},x_{6}]=x_{11}, 
\]%
together with relations which imply that $x_{10},x_{11},\ldots .x_{23}$ are
central. Call this covering algebra $M$. Then $M$ has dimension 23, and the
nucleus of $M$ is $[M,M,M]$ which has dimension 2 and is spanned by $x_{10}$
and $x_{11}$ (with $x_{10}=[x_{4},x_{1},x_{2}]$ and $%
x_{11}=[x_{4},x_{1},x_{5}]$). The immediate descendants of $L_{p}$ are
algebras $M/I$, where $I$ is a proper subspace of Sp$\langle
x_{10},x_{11},\ldots ,x_{23}\rangle $ such that 
\[
I+\text{Sp}\langle x_{10},x_{11}\rangle =\text{Sp}\langle
x_{10},x_{11},\ldots ,x_{23}\rangle . 
\]%
Thus $L_{p}$ has immediate descendants of dimension 10 and 11.

\section{Descendants of $L_{p}$ of dimension \ 10}

Let $\left[ 
\begin{array}{cc}
\alpha & \beta \\ 
\gamma & \delta%
\end{array}%
\right] \in \,$GL$(2,p)$, and let%
\begin{eqnarray*}
y_{1} &=&\alpha x_{1}+\beta x_{4}, \\
y_{2} &=&\alpha x_{2}+\beta x_{5}, \\
y_{3} &=&\alpha x_{3}+\beta x_{6}, \\
y_{4} &=&\gamma x_{1}+\delta x_{4}, \\
y_{5} &=&\gamma x_{2}+\delta x_{5}, \\
y_{6} &=&\gamma x_{3}+\delta x_{6}.
\end{eqnarray*}%
Then, $[y_{4},y_{1}]=(\alpha \delta -\beta \gamma )[x_{4},x_{1}]$, and%
\begin{eqnarray*}
\lbrack y_{4},y_{1},y_{2}] &=&\alpha (\alpha \delta -\beta \gamma
)x_{10}+\beta (\alpha \delta -\beta \gamma )x_{11}, \\
\lbrack y_{4},y_{1},y_{5}] &=&\gamma (\alpha \delta -\beta \gamma
)x_{10}+\delta (\alpha \delta -\beta \gamma )x_{11}.
\end{eqnarray*}%
This means that if $M/I$ is an immediate descendant of $L_{p}$ of dimension
10, then we can choose $\left[ 
\begin{array}{cc}
\alpha & \beta \\ 
\gamma & \delta%
\end{array}%
\right] $ so that $[y_{4},y_{1},y_{2}]+I$ generates Sp$\langle
x_{10},x_{11},\ldots ,x_{23}\rangle /I$, and so that $[y_{4},y_{1},y_{5}]\in
I$. Note that if we let $J=$Sp$\langle x_{10},x_{11},\ldots ,x_{23}\rangle $
then $M/J$ is isomorphic to $L_{p}$ and the map $x_{i}+J\mapsto y_{i}+J$ for 
$i=1,2,\ldots ,6$ extends to an automorphism of $L_{p}$. So every immediate
descendant of $L_{p}$ of dimension 10 has a presentation on generators $%
x_{1},x_{2},\ldots ,x_{10}$ with relations%
\begin{equation}
\lbrack x_{2},x_{1}]=\varepsilon x_{10},
\end{equation}%
\[
\lbrack x_{3},x_{1}]=\zeta x_{10}, 
\]%
\[
\lbrack x_{3},x_{2}]=\eta x_{10}, 
\]%
\[
\lbrack x_{4},x_{1}]=x_{7}, 
\]%
\[
\lbrack x_{4},x_{2}]=x_{8}, 
\]%
\[
\lbrack x_{4},x_{3}]=x_{9}, 
\]%
\[
\lbrack x_{5},x_{1}]=x_{8}+\theta x_{10}, 
\]%
\[
\lbrack x_{5},x_{2}]=x_{7}+\kappa x_{10}, 
\]%
\[
\lbrack x_{5},x_{3}]=\lambda x_{10}, 
\]%
\[
\lbrack x_{5},x_{4}]=\mu x_{10}, 
\]%
\[
\lbrack x_{6},x_{1}]=x_{9}+\nu x_{10}, 
\]%
\[
\lbrack x_{6},x_{2}]=\xi x_{10}, 
\]%
\[
\lbrack x_{6},x_{3}]=x_{8}+\pi x_{10}, 
\]%
\[
\lbrack x_{6},x_{4}]=\rho x_{10}, 
\]%
\[
\lbrack x_{6},x_{5}]=\sigma x_{10}, 
\]%
\[
\lbrack x_{7},x_{1}]=0, 
\]%
\[
\lbrack x_{7},x_{2}]=x_{10}, 
\]%
\[
\lbrack x_{7},x_{3}]=0, 
\]%
\[
\lbrack x_{7},x_{4}]=0, 
\]%
\[
\lbrack x_{7},x_{5}]=0, 
\]%
\[
\lbrack x_{7},x_{6}]=0, 
\]%
\[
\lbrack x_{8},x_{1}]=x_{10}, 
\]%
\[
\lbrack x_{8},x_{2}]=0, 
\]%
\[
\lbrack x_{8},x_{3}]=0, 
\]%
\[
\lbrack x_{8},x_{4}]=0, 
\]%
\[
\lbrack x_{8},x_{5}]=0, 
\]%
\[
\lbrack x_{8},x_{6}]=0, 
\]%
\[
\lbrack x_{9},x_{1}]=0, 
\]%
\[
\lbrack x_{9},x_{2}]=0, 
\]%
\[
\lbrack x_{9},x_{3}]=x_{10}, 
\]%
\[
\lbrack x_{9},x_{4}]=0, 
\]%
\[
\lbrack x_{9},x_{5}]=0, 
\]%
\[
\lbrack x_{9},x_{6}]=0, 
\]%
\[
\lbrack x_{10},x_{1}]=0, 
\]%
\[
\lbrack x_{10},x_{2}]=0, 
\]%
\[
\lbrack x_{10},x_{3}]=0, 
\]%
\[
\lbrack x_{10},x_{4}]=0, 
\]%
\[
\lbrack x_{10},x_{5}]=0, 
\]%
\[
\lbrack x_{10},x_{6}]=0, 
\]%
for some scalars $\varepsilon ,\zeta ,\ldots ,\sigma $.

If we take this presentation, and let%
\begin{eqnarray*}
y_{1} &=&x_{1}, \\
y_{2} &=&x_{2}-\varepsilon x_{8}, \\
y_{3} &=&x_{3}-\eta x_{7}-\zeta x_{8}, \\
y_{4} &=&x_{4}, \\
y_{5} &=&x_{5}-\kappa x_{7}-\theta x_{8}-\lambda x_{9}, \\
y_{6} &=&x_{6}-\xi x_{7}-\nu x_{8}-\pi x_{9},
\end{eqnarray*}%
then%
\[
\lbrack y_{2},y_{1}]=[x_{2},x_{1}]-\varepsilon \lbrack x_{8},x_{1}]=0, 
\]%
\[
\lbrack y_{3},y_{1}]=[x_{3},x_{1}]-\eta \lbrack x_{7},x_{1}]-\zeta \lbrack
x_{8},x_{1}]=0, 
\]%
\[
\lbrack y_{3},y_{2}]=[x_{3},x_{2}]-\eta \lbrack x_{7},x_{2}]-\zeta \lbrack
x_{8},x_{2}]+\varepsilon \lbrack x_{8},x_{3}]=0, 
\]%
\[
\lbrack y_{4},y_{1}]=[x_{4},x_{1}]=x_{7}, 
\]%
\[
\lbrack y_{4},y_{2}]=[x_{4},x_{2}]+\varepsilon \lbrack x_{8},x_{4}]=x_{8}, 
\]%
\[
\lbrack y_{4},y_{3}]=[x_{4},x_{3}]+\eta \lbrack x_{7},x_{4}]+\zeta \lbrack
x_{8},x_{4}]=x_{9}, 
\]%
\[
\lbrack y_{5},y_{1}]=[x_{5},x_{1}]-\kappa \lbrack x_{7},x_{1}]-\theta
\lbrack x_{8},x_{1}]-\lambda \lbrack x_{9},x_{1}]=x_{8}, 
\]%
\[
\lbrack y_{5},y_{2}]=[x_{5},x_{2}]-\kappa \lbrack x_{7},x_{2}]-\theta
\lbrack x_{8},x_{2}]-\lambda \lbrack x_{9},x_{2}]+\varepsilon \lbrack
x_{8},x_{5}]=x_{7}, 
\]%
\[
\lbrack y_{5},y_{3}]=[x_{5},x_{3}]-\kappa \lbrack x_{7},x_{3}]-\theta
\lbrack x_{8},x_{3}]-\lambda \lbrack x_{9},x_{3}]+\eta \lbrack
x_{7},x_{5}]+\zeta \lbrack x_{8},x_{5}]=0, 
\]%
\[
\lbrack y_{5},y_{4}]=[x_{5},x_{4}]-\kappa \lbrack x_{7},x_{4}]-\theta
\lbrack x_{8},x_{4}]-\lambda \lbrack x_{9},x_{4}]=\mu x_{10}, 
\]%
\[
\lbrack y_{6},y_{1}]=[x_{6},x_{1}]-\xi \lbrack x_{7},x_{1}]-\nu \lbrack
x_{8},x_{1}]-\pi \lbrack x_{9},x_{1}]=x_{9}, 
\]%
\[
\lbrack y_{6},y_{2}]=[x_{6},x_{2}]-\xi \lbrack x_{7},x_{2}]-\nu \lbrack
x_{8},x_{2}]-\pi \lbrack x_{9},x_{2}]+\varepsilon \lbrack x_{8},x_{6}]=0, 
\]%
\[
\lbrack y_{6},y_{3}]=[x_{6},x_{3}]-\xi \lbrack x_{7},x_{3}]-\nu \lbrack
x_{8},x_{3}]-\pi \lbrack x_{9},x_{3}]+\kappa \lbrack x_{7},x_{6}]+\theta
\lbrack x_{8},x_{6}]+\lambda \lbrack x_{9},x_{6}]=x_{8}, 
\]%
\[
\lbrack y_{6},y_{4}]=[x_{6},x_{4}]-\xi \lbrack x_{7},x_{4}]-\nu \lbrack
x_{8},x_{4}]-\pi \lbrack x_{9},x_{4}]=\rho x_{10}, 
\]%
\[
\lbrack y_{6},y_{5}]=[x_{6},x_{5}]-\xi \lbrack x_{7},x_{5}]-\nu \lbrack
x_{8},x_{5}]-\pi \lbrack x_{9},x_{5}]+\kappa \lbrack x_{7},x_{6}]+\theta
\lbrack x_{8},x_{6}]+\lambda \lbrack x_{9},x_{6}]=\sigma x_{10}. 
\]%
And if we define $y_{7}=x_{7}$, $y_{8}=x_{8}$, $y_{9}=x_{9}$, $y_{10}=x_{10}$
then we obtain the relations 
\[
\lbrack y_{7},y_{1}]=0, 
\]%
\[
\lbrack y_{7},y_{2}]=y_{10}, 
\]%
\[
\lbrack y_{7},y_{3}]=0, 
\]%
\[
\lbrack y_{7},y_{4}]=0, 
\]%
\[
\lbrack y_{7},y_{5}]=0, 
\]%
\[
\lbrack y_{7},y_{6}]=0, 
\]%
\[
\lbrack y_{8},y_{1}]=y_{10}, 
\]%
\[
\lbrack y_{8},y_{2}]=0, 
\]%
\[
\lbrack y_{8},y_{3}]=0, 
\]%
\[
\lbrack y_{8},y_{4}]=0, 
\]%
\[
\lbrack y_{8},y_{5}]=0, 
\]%
\[
\lbrack y_{8},y_{6}]=0, 
\]%
\[
\lbrack y_{9},y_{1}]=0, 
\]%
\[
\lbrack y_{9},y_{2}]=0, 
\]%
\[
\lbrack y_{9},y_{3}]=y_{10}, 
\]%
\[
\lbrack y_{9},y_{4}]=0, 
\]%
\[
\lbrack y_{9},y_{5}]=0, 
\]%
\[
\lbrack y_{9},y_{6}]=0, 
\]%
\[
\lbrack y_{10},y_{1}]=0, 
\]%
\[
\lbrack y_{10},y_{2}]=0, 
\]%
\[
\lbrack y_{10},y_{3}]=0, 
\]%
\[
\lbrack y_{10},y_{4}]=0, 
\]%
\[
\lbrack y_{10},y_{5}]=0, 
\]%
\[
\lbrack y_{10},y_{6}]=0. 
\]

It follows that every immediate descendant of $L_{p}$ of dimension 10 has a
presentation on generators $x_{1},x_{2},\ldots ,x_{10}$ with relations%
\[
\lbrack x_{4},x_{1}]=x_{7}, 
\]%
\[
\lbrack x_{4},x_{2}]=x_{8}, 
\]%
\[
\lbrack x_{4},x_{3}]=x_{9}, 
\]%
\[
\lbrack x_{5},x_{1}]=x_{8}, 
\]%
\[
\lbrack x_{5},x_{2}]=x_{7}, 
\]%
\[
\lbrack x_{5},x_{4}]=\mu x_{10}, 
\]%
\[
\lbrack x_{6},x_{1}]=x_{9}, 
\]%
\[
\lbrack x_{6},x_{3}]=x_{8}, 
\]%
\[
\lbrack x_{6},x_{4}]=\rho x_{10}, 
\]%
\[
\lbrack x_{6},x_{5}]=\sigma x_{10}, 
\]%
\[
\lbrack x_{7},x_{2}]=x_{10}, 
\]%
\[
\lbrack x_{8},x_{1}]=x_{10}, 
\]%
\[
\lbrack x_{9},x_{3}]=x_{10}, 
\]%
for some scalars $\mu ,\rho ,\sigma $, and with all other commutators $%
[x_{i},x_{j}]$ with $i>j$ trivial.

\section{Counting the descendants of dimension 10}

As we showed above, every immediate descendant of $L_{p}$ of dimension 10
has a presentation on generators $x_{1},x_{2},\ldots ,x_{10}$ with relations%
\[
\lbrack x_{4},x_{1}]=x_{7}, 
\]%
\[
\lbrack x_{4},x_{2}]=x_{8}, 
\]%
\[
\lbrack x_{4},x_{3}]=x_{9}, 
\]%
\[
\lbrack x_{5},x_{1}]=x_{8}, 
\]%
\[
\lbrack x_{5},x_{2}]=x_{7}, 
\]%
\[
\lbrack x_{5},x_{4}]=\lambda x_{10}, 
\]%
\[
\lbrack x_{6},x_{1}]=x_{9}, 
\]%
\[
\lbrack x_{6},x_{3}]=x_{8}, 
\]%
\[
\lbrack x_{6},x_{4}]=\mu x_{10}, 
\]%
\[
\lbrack x_{6},x_{5}]=\nu x_{10}, 
\]%
\[
\lbrack x_{7},x_{2}]=x_{10}, 
\]%
\[
\lbrack x_{8},x_{1}]=x_{10}, 
\]%
\[
\lbrack x_{9},x_{3}]=x_{10}, 
\]%
for some scalars $\lambda ,\mu ,\nu $, and with all other commutators $%
[x_{i},x_{j}]$ with $i>j$ trivial. Denote this algebra by $A_{(\lambda ,\mu
,\nu )}$. The isomorphism type of $A_{(\lambda ,\mu ,\nu )}$ is determined
by the triple $(\lambda ,\mu ,\nu )$, but we still need to solve the problem
of when two different triples give isomorphic algebras. Suppose that $%
A_{(\lambda ,\mu ,\nu )}$ is isomorphic to $A_{(\lambda ^{\prime },\mu
^{\prime },\nu ^{\prime })}$, and let $\theta :A_{(\lambda ^{\prime },\mu
^{\prime },\nu ^{\prime })}\rightarrow A_{(\lambda ,\mu ,\nu )}$ be an
isomorphism. Let $y_{1},y_{2},\ldots ,y_{6}$ be the images in $A_{(\lambda
,\mu ,\nu )}$ under $\theta $ of the defining generators of $A_{(\lambda
^{\prime },\mu ^{\prime },\nu ^{\prime })}$. Note that $A_{(\lambda ,\mu
,\nu )}/\langle x_{10}\rangle $ is isomorphic to $L_{p}$, and that the map $%
x_{i}+\langle x_{10}\rangle \mapsto y_{i}+\langle x_{10}\rangle $ $%
(i=1,2,\ldots ,6)$ extends to an automorphism of $L_{p}$. Note also that 
\begin{eqnarray*}
&&C_{A_{(\lambda ,\mu ,\nu )}}([A_{(\lambda ,\mu ,\nu )},A_{(\lambda ,\mu
,\nu )}]) \\
&=&[A_{(\lambda ,\mu ,\nu )},A_{(\lambda ,\mu ,\nu )}]+\text{Sp}\langle
x_{4},x_{5},x_{6}\rangle \\
&=&[A_{(\lambda ,\mu ,\nu )},A_{(\lambda ,\mu ,\nu )}]+\text{Sp}\langle
y_{4},y_{5},y_{6}\rangle .
\end{eqnarray*}%
It follows that $A_{(\lambda ,\mu ,\nu )}$ is isomorphic to $A_{(\lambda
^{\prime },\mu ^{\prime },\nu ^{\prime })}$ if and only if $A_{(\lambda ,\mu
,\nu )}$ has a set of generators $y_{1},y_{2},\ldots ,y_{6}$ satisfying the
defining relations of $A_{(\lambda ^{\prime },\mu ^{\prime },\nu ^{\prime
})} $, and that this can only happen if the map $x_{i}+\langle x_{10}\rangle
\mapsto y_{i}+\langle x_{10}\rangle $ $(i=1,2,\ldots ,6)$ extends to an
automorphism of $L_{p}$, and if%
\begin{equation}
\lbrack A_{(\lambda ,\mu ,\nu )},A_{(\lambda ,\mu ,\nu )}]+\text{Sp}\langle
x_{4},x_{5},x_{6}\rangle =[A_{(\lambda ,\mu ,\nu )},A_{(\lambda ,\mu ,\nu
)}]+\text{Sp}\langle y_{4},y_{5},y_{6}\rangle .
\end{equation}

The first thing to observe is that if we let $y_{1}=x_{1}$, $y_{2}=x_{2}$, $%
y_{3}=x_{3}$, $y_{4}=\delta x_{4}$, $y_{5}=\delta x_{5}$, $y_{6}=\delta
x_{6} $ in $A_{(\lambda ,\mu ,\nu )}$, then $y_{1},y_{2},\ldots ,y_{6}$
satisfy the defining relations of $A_{(\delta \lambda ,\delta \mu ,\delta
\nu )}$. (This is easy to check.) So the triples $(\lambda ,\mu ,\nu )$ and $%
(\delta \lambda ,\delta \mu ,\delta \nu )$ define isomorphic algebras, and
the isomorphism type of $A_{(\lambda ,\mu ,\nu )}$ depends only on the
ratios $\lambda :\mu $, $\lambda :\nu $, $\mu :\nu $. The next thing to note
is that if $y_{1},y_{2},\ldots ,y_{6}\in A_{(\lambda ,\mu ,\nu )}$ satisfy
the defining relations of $A_{(\lambda ^{\prime },\mu ^{\prime },\nu
^{\prime })} $ then the ratios $\lambda ^{\prime }:\mu ^{\prime }$, $\lambda
^{\prime }:\nu ^{\prime }$, $\mu ^{\prime }:\nu ^{\prime }$ depend only on
the values of $y_{4},y_{5},y_{6}$, and not on the values of $%
y_{1},y_{2},y_{3}$. The calculations in Section 7, together with equation
(4) and the fact that the map $x_{i}+\langle x_{10}\rangle \mapsto
y_{i}+\langle x_{10}\rangle $ $(i=1,2,\ldots ,6)$ extends to an automorphism
of $L_{p}$, imply that%
\[
\left[ \;%
\begin{array}{c}
y_{4} \\ 
y_{5} \\ 
y_{6}%
\end{array}%
\right] =\delta A\left[ \;%
\begin{array}{c}
x_{4} \\ 
x_{5} \\ 
x_{6}%
\end{array}%
\right] +\left[ \;%
\begin{array}{c}
b_{1} \\ 
b_{2} \\ 
b_{3}%
\end{array}%
\right] 
\]%
where $b_{1},b_{2},b_{3}\in \lbrack A_{(\lambda ,\mu ,\nu )},A_{(\lambda
,\mu ,\nu )}]$, where $\delta \neq 0$, and where%
\[
A=\left[ 
\begin{array}{ccc}
u & 0 & 0 \\ 
0 & u^{-1} & 0 \\ 
0 & 0 & 1%
\end{array}%
\right] 
\]%
with $u^{4}=1$, or%
\[
A=\left[ 
\begin{array}{ccc}
a & ab & ac \\ 
df & -f & -def \\ 
1 & d & e%
\end{array}%
\right] , 
\]%
as described in Section 7. Furthermore, since $x_{4},x_{5},x_{6}$ centralize 
$[A_{(\lambda ,\mu ,\nu )},A_{(\lambda ,\mu ,\nu )}]$ the values of $%
[y_{5},y_{4}]$, $[y_{6},y_{4}]$, $[y_{6},y_{5}]$ depend only on $\delta A$,
and not on $b_{1},b_{2},b_{3}$.

We now show that%
\[
\left[ \;%
\begin{array}{c}
y_{4} \\ 
y_{5} \\ 
y_{6}%
\end{array}%
\right] =\delta A\left[ \;%
\begin{array}{c}
x_{4} \\ 
x_{5} \\ 
x_{6}%
\end{array}%
\right] 
\]%
can arise for all $\delta A$ of the form just described. Specifically, we
show that if we set%
\[
\left[ \;%
\begin{array}{c}
y_{1} \\ 
y_{2} \\ 
y_{3}%
\end{array}%
\right] =\alpha A\left[ \;%
\begin{array}{c}
x_{1} \\ 
x_{2} \\ 
x_{3}%
\end{array}%
\right] ,\;\left[ \;%
\begin{array}{c}
y_{4} \\ 
y_{5} \\ 
y_{6}%
\end{array}%
\right] =\delta A\left[ \;%
\begin{array}{c}
x_{4} \\ 
x_{5} \\ 
x_{6}%
\end{array}%
\right] 
\]%
where $\alpha ,\delta \neq 0$, and where $A$ is as just described, then $%
y_{1},y_{2},\ldots ,y_{6}$ \emph{do} satisfy the defining relations of the
algebra $A_{(\lambda ^{\prime },\mu ^{\prime },\nu ^{\prime })}$, for some $%
(\lambda ^{\prime },\mu ^{\prime },\nu ^{\prime })$ which we will determine
below. One possible way of checking this is to compute $[y_{i},y_{j}]$ in
terms of $x_{7},x_{8},x_{9},x_{10}$ for all $i>j$, and check all the
relations one by one. But there is a shortcut. We know that the map $%
x_{i}+\langle x_{10}\rangle \mapsto y_{i}+\langle x_{10}\rangle $ $%
(i=1,2,\ldots ,6)$ extends to an automorphism of $L_{p}$. We also know that $%
y_{4},y_{5},y_{6}$ centralize $[A_{(\lambda ,\mu ,\nu )},A_{(\lambda ,\mu
,\nu )}]$. In particular $[y_{4},y_{1},y_{5}]=0$. So if we set $%
y_{7}=[y_{4},y_{1}]$, $y_{8}=[y_{4},y_{2}]$, $y_{9}=[y_{4},y_{3}]$, $%
y_{10}=[y_{4},y_{1},y_{2}]$, then $y_{1},y_{2},\ldots ,y_{10}$ must satisfy
relations of the form (3) for some scalars $\varepsilon ,\zeta ,\ldots
,\sigma $. However we must have $\varepsilon =\zeta =\eta =0$ since the
linear span of $y_{1},y_{2},y_{3}$ is the same as the linear span of $%
x_{1},x_{2},x_{3}$, and 
\[
\lbrack x_{2},x_{1}]=[x_{3},x_{1}]=[x_{3},x_{2}]=0. 
\]%
We must also have $\theta =\kappa =\lambda =\nu =\xi =\pi =0$ since if $%
4\leq r\leq 6$ and $1\leq s\leq 3$ then 
\[
\lbrack y_{r},y_{s}]\in \text{Sp}\langle \lbrack x_{i},x_{j}]\,|\,i\in
\{4,5,6\},\;j\in \{1,2,3\}\}=\text{Sp}\langle x_{7},x_{8},x_{9}\rangle . 
\]

It remains to compute $[y_{5},y_{4}]$, $[y_{6},y_{4}]$, $[y_{6},y_{5}]$.

First consider the case when 
\[
A=\left[ 
\begin{array}{ccc}
u & 0 & 0 \\ 
0 & u^{-1} & 0 \\ 
0 & 0 & 1%
\end{array}%
\right] . 
\]%
Then%
\begin{eqnarray*}
\lbrack y_{5},y_{4}] &=&\delta ^{2}[x_{5},x_{4}]=\delta ^{2}\lambda x_{10},
\\
\lbrack y_{6},y_{4}] &=&\delta ^{2}u[x_{6},x_{4}]=\delta ^{2}u\mu x_{10}, \\
\lbrack y_{6},y_{5}] &=&\delta ^{2}u^{-1}[x_{6},x_{5}]=\delta ^{2}u^{-1}\nu
x_{10}.
\end{eqnarray*}%
Now 
\[
y_{10}=[y_{4},y_{1},y_{2}]=\alpha ^{2}\delta u[x_{4},x_{1},x_{2}]=\alpha
^{2}\delta ux_{10}, 
\]
and so $y_{1},y_{2},\ldots ,y_{10}$ satsify the defining relations of $%
A_{(k\lambda ,ku\mu ,ku^{-1}\nu )}$ where $k=\alpha ^{-2}\delta u^{-1}$.
Note that as $\alpha $ and $\delta $ take on all possible non-zero values, $%
k $ takes on all possible non-zero values.

Next consider the case when%
\[
A=\left[ 
\begin{array}{ccc}
a & ab & ac \\ 
df & -f & -def \\ 
1 & d & e%
\end{array}%
\right] . 
\]%
Then%
\begin{eqnarray*}
\lbrack y_{5},y_{4}] &=&\delta ^{2}\left(
-(af+abdf)[x_{5},x_{4}]-(adef+acdf)[x_{6},x_{4}]-(abdef-acf)[x_{6},x_{5}]%
\right) , \\
\lbrack y_{6},y_{4}] &=&\delta ^{2}\left(
(ad-ab)[x_{5},x_{4}]+(ae-ac)[x_{6},x_{4}]+(abe-acd)[x_{6},x_{5}]\right) , \\
\lbrack y_{6},y_{5}] &=&\delta ^{2}\left(
(d^{2}f+f)[x_{5},x_{4}]+2def[x_{6},x_{4}]-(ef-d^{2}ef)[x_{6},x_{5}]\right) .
\end{eqnarray*}%
So $y_{1},y_{2},\ldots ,y_{10}$ satsify the defining relations of $%
A_{(\lambda ^{\prime },\mu ^{\prime },\nu ^{\prime })}$ where%
\[
\left[ \;%
\begin{array}{c}
\lambda ^{\prime } \\ 
\mu ^{\prime } \\ 
\nu ^{\prime }%
\end{array}%
\right] =k\left[ 
\begin{array}{ccc}
-abdf-af & -acdf-adef & -abdef+acf \\ 
-ab+ad & -ac+ae & abe-acd \\ 
d^{2}f+f & 2def & d^{2}ef-ef%
\end{array}%
\right] \left[ \;%
\begin{array}{c}
\lambda \\ 
\mu \\ 
\nu%
\end{array}%
\right] , 
\]%
for some non-zero scalar $k$ which takes on all possible values as $\alpha $
and $\delta $ take on all possible values. Substituting the solutions for $%
a,c,d,f$ in terms of $d$ and $e$, and using the fact that $e^{2}=\frac{%
d^{2}-1}{d}$, we obtain%
\[
\left[ \;%
\begin{array}{c}
\lambda ^{\prime } \\ 
\mu ^{\prime } \\ 
\nu ^{\prime }%
\end{array}%
\right] =k\left[ 
\begin{array}{ccc}
\frac{1}{2d}\left( d^{2}+1\right) ^{2} & -e\left( d^{2}+1\right) & \frac{1}{%
2d}e\left( d^{4}+4d^{2}+3\right) \\ 
\frac{1}{2}du\frac{e}{d^{2}-1}\left( d^{2}+1\right) ^{2} & -\frac{u}{d}%
\left( d^{2}+1\right) & -\frac{1}{2}u\left( d^{4}+4d^{2}+3\right) \\ 
2u^{-1}e & 4u^{-1}\frac{d^{2}-1}{d^{2}+1} & \frac{2u^{-1}}{d}\frac{\left(
d^{2}-1\right) ^{2}}{d^{2}+1}%
\end{array}%
\right] \left[ \;%
\begin{array}{c}
\lambda \\ 
\mu \\ 
\nu%
\end{array}%
\right] . 
\]

So we have an action on $\mathbb{F}_{p}^{3}$ of the form%
\begin{equation}
\left[ \;%
\begin{array}{c}
\lambda \\ 
\mu \\ 
\nu%
\end{array}%
\right] \rightarrow kB\left[ \;%
\begin{array}{c}
\lambda \\ 
\mu \\ 
\nu%
\end{array}%
\right] ,
\end{equation}%
where $k$ is an arbitrary non-zero scalar, and where $B$ is a matrix of the
form%
\begin{equation}
B=\left[ 
\begin{array}{ccc}
1 & 0 & 0 \\ 
0 & u & 0 \\ 
0 & 0 & u^{-1}%
\end{array}%
\right]
\end{equation}%
(with $u^{4}=1$) or a matrix of the form%
\begin{equation}
B=\left[ 
\begin{array}{ccc}
\frac{1}{2d}\left( d^{2}+1\right) ^{2} & -e\left( d^{2}+1\right) & \frac{1}{%
2d}e\left( d^{4}+4d^{2}+3\right) \\ 
\frac{1}{2}du\frac{e}{d^{2}-1}\left( d^{2}+1\right) ^{2} & -\frac{u}{d}%
\left( d^{2}+1\right) & -\frac{1}{2}u\left( d^{4}+4d^{2}+3\right) \\ 
2u^{-1}e & 4u^{-1}\frac{d^{2}-1}{d^{2}+1} & \frac{2u^{-1}}{d}\frac{\left(
d^{2}-1\right) ^{2}}{d^{2}+1}%
\end{array}%
\right]
\end{equation}%
(with $d$ and $e$ solutions of $d^{4}+6d^{2}-3=0$ and $1-d^{2}+de^{2}=0$ and
with $u^{4}=1$). The actual matrices that occur depend on the residue class
of $p$ modulo 12. If $p=1\mod{12}$ then we have 4 matrices of the form
(6) and either 0 or 32 matrices of the form (7). So when $p=1\mod{12}$
we either have a group of order $4(p-1)$ acting on $\mathbb{F}_{p}^{3}$, or
we have a group of order $36(p-1)$. If $p=5\mod{12}$ then we have 4
matrices of the form (6) and none of the form (7). So we have a group of
order $4(p-1)$ acting on $\mathbb{F}_{p}^{3}$. If $p=7\mod{12}$ then
there are 2 matrices of the form (6) and none of the form (7), so we have a
group of order $2(p-1)$ acting on $\mathbb{F}_{p}^{3}$. Finally, if $p=11%
\mod{12}$ then we have 2 matrices of the form (6) and 4 matrices of the
form (7), so that we have a group of order $6(p-1)$ acting on $\mathbb{F}%
_{p}^{3}$.

The number of isomorphism classes of algebras $A_{(\lambda ,\mu ,\nu )}$ is
the number of orbits in the action of these groups on $\mathbb{F}_{p}^{3}$.
We compute the number of orbits in each case by computing the number of
vectors in $\mathbb{F}_{p}^{3}$ fixed by each transformation of the form
(5). First note that all the transformations fix $\left[ \;%
\begin{array}{c}
0 \\ 
0 \\ 
0%
\end{array}%
\right] $. On the other hand, a non-zero vector $\left[ \;%
\begin{array}{c}
\lambda \\ 
\mu \\ 
\nu%
\end{array}%
\right] $ can only be fixed by a transformation of the form (5) if it is an
eigenvector of $B$, and in that case it is fixed if and only if $k$ is the
multiplicative inverse of the eigenvalue. So we need to count the (non-zero)
eigenvectors for each of the matrices $B$. A matrix of the form (6) has $%
p^{3}-1$ eigenvectors if $u=1$, $p^{2}+p-2$ eigenvectors if $u=-1$, and $%
3p-3 $ eigenvectors if $u^{2}=-1$. If $u=1$ then a matrix of the form (7)
has characteristic polynomial $x^{3}-\frac{256}{3}(d^{3}-d)$, and an
eigenvalue $-\frac{4}{3}(d^{3}+3d)$. So if $p=11\mod{12}$ then the
matrix has a single eigenvalue of multiplicity 1, and $p-1$ eigenvectors.
But if $p=1\mod{12}$ then the matrix has 3 distinct eigenvalues and $%
3p-3 $ eigenvectors. If $u=-1$ then a matrix of the form (7) is
diagonalizable with eigenvalues $\frac{4}{3}(d^{3}+3d)$, $\frac{4}{3}%
(d^{3}+3d)$, $-\frac{4}{3}(d^{3}+3d)$, and $p^{2}+p-2$ eigenvectors.
Finally, if $u^{2}=-1$ then a matrix of the form (7) has 3 distinct
eigenvalues $-4du+\frac{2}{3}(d^{3}+3d) $, $4d+\frac{2}{3}(d^{3}+3d)u$, $-4d-%
\frac{2}{3}(d^{3}+3d)u$, and so has $3p-3$ eigenvectors.

It follows that if $p=1\mod{12}$ and if there are no solutions to the
equations $d^{4}+6d^{2}-3=0$ and $1-d^{2}+de^{2}=0$, or if $p=5\mod{12}$%
, then the number of orbits (i.e. the number of descendants of $L_{p}$ of
dimension 10) is%
\[
\frac{4(p-1)+(p^{3}-1)+(p^{2}+p-2)+2(3p-3)}{4(p-1)}=\frac{(p+1)^{2}}{4}+3. 
\]%
If $p=1\mod{12}$ and if there are solutions to the equations $%
d^{4}+6d^{2}-3=0$ and $1-d^{2}+de^{2}=0$, then the number of orbits is%
\begin{eqnarray*}
&&\frac{36(p-1)+(p^{3}-1)+(p^{2}+p-2)+2(3p-3)+8(p^{2}+p-2)+24(3p-3)}{36(p-1)}
\\
&=&\frac{(p-1)^{2}}{36}+\frac{p-1}{3}+4.
\end{eqnarray*}%
If $p=7\mod{12}$ then the number of orbits is%
\[
\frac{2(p-1)+(p^{3}-1)+(p^{2}+p-2)}{2(p-1)}=\frac{(p+1)^{2}}{2}+2. 
\]%
And finally if $p=11\mod{12}$ then the number of orbits is%
\[
\frac{6(p-1)+(p^{3}-1)+3(p^{2}+p-2)+2(p-1)}{6(p-1)}=\frac{(p+1)^{2}}{6}+%
\frac{p+1}{3}+2. 
\]

This completes the proof of Theorem 1.

\end{document}